\documentclass[11pt,reqno]{amsart}
\usepackage{setspace}
\usepackage{amsmath, amssymb, amscd, amsthm, amsfonts}
\usepackage{graphicx}
\usepackage{enumitem}
\usepackage{mathdots}
\usepackage{mathpazo}

\usepackage{indentfirst}
\usepackage{amssymb,amsmath,amsfonts,amsthm}
\usepackage[all]{xy}
\usepackage{enumerate}
\usepackage{mathrsfs}
\usepackage{graphicx}
\numberwithin{equation}{section}
\usepackage[all]{xy}
\usepackage{amssymb,amscd,amsthm,amsmath,graphicx,color}
\usepackage{mathrsfs}
\usepackage[english]{babel}
\usepackage[utf8x]{inputenc}

\newtheorem{theorem}{Theorem}[section]
\newtheorem*{proposition*}{Proposition}
\newtheorem{proposition}[theorem]{Proposition}
\newtheorem{corollary}[theorem]{Corollary}
\newtheorem{lemma}[theorem]{Lemma}
\newtheorem{example}[theorem]{Example}
\newtheorem{remark}[theorem]{Remark}
\newtheorem{definition}[theorem]{Definition}
\newtheorem{question}[theorem]{Question}
\newtheorem{conjecture}[theorem]{Conjecture}

\DeclareMathOperator{\Hom}{Hom}
\DeclareMathOperator{\Ext}{Ext}
\DeclareMathOperator{\Tor}{Tor}

\DeclareMathOperator{\Supp}{Supp}
\DeclareMathOperator{\depth}{depth}
\DeclareMathOperator{\cidim}{CI-dim}
\DeclareMathOperator{\pd}{pd}

\DeclareMathOperator{\gdim}{G-dim}

\DeclareMathOperator{\Tr}{Tr}
\DeclareMathOperator{\Spec}{Spec}

\DeclareMathOperator{\X}{X}
\DeclareMathOperator{\height}{ht}

\DeclareMathOperator{\NF}{NF}

\begin{document} 

\title[Vanishing of (co)homology, freeness, and Auslander-Reiten conjecture]{Vanishing of (co)homology, freeness criteria, and the Auslander-Reiten conjecture for Cohen-Macaulay Burch rings} 

\author{Rafael Holanda}
\address{Departamento de Matemática, Universidade Federal da Paraíba - 58051-900, João Pessoa, PB, Brazil}
\email{rfh@academico.ufpb.br, rf.holanda@gmail.com}

\author{Cleto B. Miranda-Neto}
\address{Departamento de Matemática, Universidade Federal da Paraíba - 58051-900, João Pessoa, PB, Brazil}
\email{cleto@mat.ufpb.br}

\date{\today}

\keywords{Vanishing of (co)homology, freeness criteria, Auslander-Reiten conjecture, Huneke-Wiegand conjecture, Zariski-Lipman conjecture.}
\subjclass[2020]{Primary 13D05, 13D07, 13H05, 13C10, 13C14; Secondary 13H10, 13D02}


\maketitle

\begin{abstract} We establish new results on (co)homology vanishing and Ext-Tor dualities, and derive a number of freeness criteria for finite modules over Cohen-Macaulay local rings. In the main application, we settle the long-standing Auslander-Reiten conjecture for the class of Cohen-Macaulay Burch rings, among other results toward this and related problems, e.g., the Tachikawa and Huneke-Wiegand conjectures. We also derive results on further topics of interest such as Cohen-Macaulayness of tensor products and Tor-independence, and inspired by a paper of Huneke and Leuschke we obtain characterizations of when a local ring is regular, or a complete intersection, or Gorenstein; for the regular case, we describe progress on some classical differential problems, e.g., the strong version of the Zariski-Lipman conjecture. Along the way, we generalize several results from the literature and propose various questions.
\end{abstract}

\section{Motivation: Three celebrated homological conjectures}

One of the most celebrated problems in homological commutative algebra is the following commutative local version of the conjecture (initially raised for Artin algebras and known to be equivalent to the so-called Generalized Nakayama Conjecture) from \cite[p.\,70]{AR}.

\begin{conjecture}[Auslander-Reiten]\label{arconjecture}
Let $R$ be a $($Noetherian, commutative, unital$)$ local ring. If $M$ is a finitely generated $R$-module such that $\Ext^j_R(M,M)=\Ext^j_R(M,R)=0$ for all $j>0$, then $M$ is free.
\end{conjecture}

This problem has been extensively explored in the literature but remains open even if $R$ is Gorenstein, among a number of other situations. Some of the known cases can be found in 
\cite{A,  ACST, AY, AB, GT, HL, HSV, KOT, STAK}.

By considering Cohen-Macaulay local rings possessing a canonical (i.e., dualizing) module, the Auslander-Reiten conjecture is immediately seen to imply the following version (suggested in \cite{AB2}; see also \cite{HH}) of the long-standing Tachikawa conjecture.  

\begin{conjecture}[Tachikawa]\label{tachikawanconjecture}
Let $R$ be a $($Noetherian, commutative, unital$)$ Cohen-Macaulay local ring with canonical module $\omega_R$. If $\Ext^j_R(\omega_R,R)=0$ for all $j>0$, then $R$ is Gorenstein.
\end{conjecture}

Next, we invoke yet another famous conjecture concerning Ext vanishing. It was proposed in \cite{HW} and later reformulated in \cite{HIW}.

\begin{conjecture}[Huneke-Wiegand]\label{hiwconjecture} Let $R$ be a $($Noetherian, commutative, unital$)$ one-dimensional Gorenstein local domain. If $M$ is a finitely generated torsionfree $R$-module such that $\Ext^1_R(M, M)=0$, then $M$ is free.
\end{conjecture}

Notice that these three conjectures are particular manifestations of a more general, important line of research which remains extremely active in homological commutative algebra, namely, that of giving criteria or characterizing  freeness of modules over local rings by means of (co)homology vanishing. 



In the present paper, and in regard to the above problems, our first main contribution is to confirm the Auslander-Reiten conjecture (hence the Tachikawa conjecture) for Cohen-Macaulay Burch rings; here let us recall that the class of Burch rings was introduced in \cite{DKT}, where in particular it was shown that such a class contains, e.g., all hypersurface local rings, all Cohen-Macaulay local rings having minimal multiplicity and infinite residue field, and all Cohen-Macaulay quotients (localized at the homogeneous maximal ideal) of graded polynomial rings over an infinite field by homogeneous ideals having a linear resolution. Furthermore, we obtain a result toward the Huneke-Wiegand conjecture which depends on the vanishing of one additional Ext module but requires fewer conditions on the ring. Such results are presented in Section \ref{conjecturessec}. In Section \ref{freenesssec}, we establish a number of new freeness criteria for modules (over Cohen-Macaulay local rings), scattered along its three subsections.

The above-mentioned achievements are, essentially, byproducts of our main technical results, which rely on spectral sequence methods and are given in Section \ref{mainsect}. They deal mainly with the vanishing of Ext and Tor modules and relations to certain dualities. In the same section we also give the first applications, which concern topics of high current interest such as 
Cohen-Macaulayness (e.g., of tensor products) and Tor-independence.

In addition, we provide a last section containing further applications, namely, we use our freeness criteria to give cohomological characterizations of regular, complete intersection, and Gorenstein local rings. In our study of regular rings, we detect several connections to long-standing conjectures involving modules of derivations and differentials, such as the strong version of the Zariski-Lipman conjecture (see \cite{Her}, also \cite{VC}) and Berger's conjecture (see \cite{Ber}, also \cite{BeSur}).

Throughout the paper, we generalize and improve a number of
results from the literature, and we propose several questions which, in our view, can be of interest.

\medskip

\noindent {\it Conventions and some notations.} By a \textit{ring} we mean a commutative unitary Noetherian ring. By a {\it finite} $R$-module we mean a finitely generated $R$-module. The \textit{algebraic dual} of an $R$-module $M$ is $M^*=\Hom_R(M,R)$. Whenever $R$ is local, we write $\mathfrak{m}$ for its maximal ideal. In this case, the \emph{Matlis dual} of $M$ is written $M^\vee$. If in addition $R$ admits a canonical module $\omega_R$, we denote the \emph{canonical dual} of $M$ by $M^\dagger=\Hom_R(M, \omega_R)$. The depth of finite $R$-modules is always meant to be the ${\mathfrak m}$-depth. The projective dimension (resp. injective dimension) of a finite $R$-module $M$ is written ${\rm pd}_RM$ (resp. ${\rm id}_RM$). The height of an ideal ${\mathfrak a}$ of $R$ is denoted $\height {\mathfrak a}$. For background on commutative and homological algebra -- e.g., the (maximal) Cohen-Macaulay property, canonical modules, Ext and Tor functors, and local duality -- we refer to \cite{BH}.

\section{Vanishing of (co)homology and dualities (and first applications)}\label{mainsect}

In this section, we prove our main technical result on Ext and Tor vanishing, which in particular provides new duality results and extends the statements given in \cite[Lemma 3.4]{LM} (which is one of the main technical tools therein). We point out that, although a simpler proof for \cite[Lemma 3.4(1)]{LM} was given recently in \cite[Corollary 2.4]{KOT}, our method also provides more information about the (co)homology modules involved. Moreover, it should be noticed that item $(ii)$ below generalizes \cite[Lemma 2.3]{HH}, which can be retrieved by taking $n=1$. In addition, our result furnishes, as a bonus, two five-term exact sequences involving suitable Ext and Tor modules. 


The first applications of our theorem are also presented in this section and are concerned with relevant topics such as Cohen-Macaulayness (e.g., of tensor products), Tor-independence and the so-called depth formula. Numerous other applications will be described throughout the paper.


\begin{theorem}\label{lmgeneralization}
Let $R$ be a Cohen-Macaulay local ring of dimension $d$ and having a canonical module $\omega_R$. Let $M$ and $N$ be finite $R$-modules and suppose $N$ is maximal Cohen-Macaulay. Let $n$ be a non-negative integer. The following statements hold:
\begin{itemize}
    \item [(i)] If $\Ext^j_R(M,N)=0$ for all $j= 1, \ldots, d$, then there is an isomorphism $M\otimes_RN^\dagger\cong\Hom_R(M,N)^\dagger$, and the $R$-modules $M\otimes_RN^\dagger$ and $\Hom_R(M,N)$ are maximal Cohen-Macaulay;
    
    \item [(ii)] If $\Ext^j_R(M,N)=0$ for all $j= 1, \ldots, d+n$, then
    $$\Ext^d_R(\Ext^{d+n+1}_R(M,N),\omega_R)\cong\Tor_{n+1}^R(M,N^\dagger)$$ and, if $n\geq 1$, $\Tor^R_i(M,N^\dagger)=0$ for all $i= 1, \ldots, n$. Moreover, there exists an exact sequence
    $$\xymatrix@=1em{\Tor^R_{n+3}(M,N^\dagger)\ar[r] & \Ext^{d-2}_R(\Ext^{d+n+1}_R(M,N),\omega_R)\ar[r] & \Ext^d_R(\Ext^{d+n+2}_R(M,N),\omega_R)\ar[dl] \\ & \Tor^R_{n+2}(M,N^\dagger)\ar[r] & \Ext^{d-1}_R(\Ext^{d+n+1}_R(M,N),\omega_R)\ar[r] & 0;}$$

    \item [(iii)] If $\Tor_j^R(M,N^\dagger)=0$ for all $j= 1, \ldots, d+n$, then
    $$\Tor_{d+n+1}^R(M,N^\dagger)^\dagger\cong\Ext^{d+n+1}_R(M,N),$$
    $$\Ext^i_R(M\otimes_R N^\dagger,\omega_R)\cong\Ext^i_R(M,N) \quad \mbox{for \,all} \quad i= 0, \ldots, d$$ 
    and, if $n\geq 1$, $\Ext^k_R(M,N)=0$ for all $k= d+1, \ldots, d+n$. Also, there is an exact sequence
    $$\xymatrix@=1em{
    0\ar[r] & \Ext^1_R(\Tor_{d+n+1}^R(M,N^\dagger),\omega_R)\ar[r] & \Ext^{d+n+2}_R(M,N)\ar[r] & \Tor^R_{d+n+2}(M,N^\dagger)^\dagger\ar[dl]
    \\
    & & \Ext^2_R(\Tor_{d+n+1}^R(M,N^\dagger),\omega_R)\ar[r] & \Ext^{d+n+3}_R(M,N).
    }$$
\end{itemize}
\end{theorem}

\begin{proof} Over $R$, pick a free resolution $F_\bullet$ of $M$ and an injective resolution $E^\bullet$ of $\omega_R$. The isomorphism of double complexes
$$F_\bullet\otimes_R\Hom_R(N,E^\bullet)\cong\Hom_R(\Hom_R(F_\bullet,N),E^\bullet)$$
gives rise to two first quadrant spectral sequences, converging to the same graded $R$-module $H$,
$$E_2^{p,q}=\Ext^p_R(\Ext^q_R(M,N),\omega_R)\Rightarrow_p H^{q-p} \ \mbox{and} \ 'E_2^{p,q}=\Tor_p^R(M,\Ext^q_R(N,\omega_R))\Rightarrow_p H^{p-q}.$$
Since $N$ is maximal Cohen-Macaulay, \cite[Proposition 3.3.3]{BH} yields $'E_2^{p,q}=0$ for $q>0$, so that $'E$ collapses and  $\Tor^R_{q-p}(M,N^\dagger)\cong{} 'E_2^{q-p,0}\cong H^{q-p}$. Thus we have a spectral sequence $$E_2^{p,q}=\Ext^p_R(\Ext^q_R(M,N),\omega_R)\Rightarrow_p\Tor^R_{q-p}(M,N^\dagger),$$ which, because of local duality, must satisfy $E_2^{p,q}=0$ for all $p>d$.


Let us use this spectral sequence to prove parts $(i)$ and $(ii)$. First notice that, if $\Ext^j_R(M,N)=0$ whenever $1\leq j\leq d +n$ (for $(i)$ we consider $n=0$), the second page of $E$ has the shape

$$\xymatrix@=1em{
E_2^{0,d+n+2} & \cdots & E_2^{d-2,d+n+2} & E_2^{d-1,d+n+2} & E_2^{d,d+n+2} & 0 & \cdots
\\
E_2^{0,d+n+1} & \cdots & E_2^{d-2,d+n+1}\ar[rru]^{d_2} & E_2^{d-1,d+n+1} & E_2^{d,d+n+1} & 0 & \cdots
\\
0 & \cdots & 0 & 0 & 0 & 0 &\cdots
\\
\vdots & \iddots & \vdots & \vdots & \vdots & \vdots & \cdots
\\
0 & \cdots & 0 & 0 & 0 & 0 & \cdots
\\
E_2^{0,0}\ar@{--}[rrrrruuu] & \cdots & E_2^{d-2,0} & E_2^{d-1,0} & E_2^{d,0} & 0 & \cdots
}$$
where the dotted line corresponds to the line $Y-X=0$ in the first quadrant.
Therefore, the convergence of $E$ ensures that $$\Hom_R(M,N)^\dagger=E_2^{0,0}\cong M\otimes_RN^\dagger$$ and $$\Ext^j_R(\Hom_R(M,N),\omega_R)=E_2^{j,0}\cong\Tor^R_{-j}(M,N^\dagger)=0 \quad \mbox{for \,all} \quad j>0,$$ which implies $(i)$. Now, to conclude that $(ii)$ holds, we observe that the homomorphism $d_2$ in the diagram above is at the corner of $E$, which gives us the five-term-type exact sequence
$$\xymatrix@=1em{\Tor_{n+3}^R(M,N^\dagger)\ar[r] & E_2^{d-2,d+n+1}\ar[r]^{d_2} & E_2^{d,d+n+2}\ar[r] & \Tor_{n+2}^R(M,N^\dagger)\ar[r] & E_2^{d-1,d+n+1}\ar[r] & 0,}$$ 
and, from the convergence, $$\Ext^d_R(\Ext^{d+n+1}_R(M,N),\omega_R)=E_2^{d,d+n+1}\cong\Tor_{n+1}^R(M,N^\dagger).$$
In case $n>0$, given $1\leq i\leq n$ and $0\leq p\leq d$, we clearly have $1\leq q\leq d+n$ for all integers $q$ such that $q-p=i$. In other words, $E_2^{p,q}=0$ whenever $q-p=i$ for $1\leq i\leq n$. Therefore, $\Tor_i^R(M,N^\dagger)=0$ for all $i= 1, \ldots, n$. This concludes the proof of $(ii)$.

Finally, we show $(iii)$. Since $N$ is maximal Cohen-Macaulay, we can consider the spectral sequence $$E_2^{p,q}=\Ext^p_R(\Tor^R_q(M,N^\dagger),\omega_R)\Rightarrow_p\Ext^{p+q}_R(M,N).$$ Along with local duality, the hypotheses force $E_2^{p,q}=0$ for $p>d$ or $1\leq q\leq d+n$. By analyzing the shape of $E_2$ as we did above for items $(i)$ and $(ii)$, we get the desired isomorphisms by convergence, and the exact sequence is exactly the five-term exact sequence of $E_2$.
\end{proof}

\begin{remark}\label{main-Art}\rm $(i)$ Let us point out that Theorem \ref{lmgeneralization}$(i)$ also holds in the situation where $R$ is an Artinian local ring, i.e., the case $d=0$. This follows by an easy inspection of the same spectral sequence, provided that we replace the hypothesis $\Ext^j_R(M,N)=0$ for all $j= 1, \ldots, d$ with the condition $\Ext^j_R(M,N)=0$ for all $j>0$. The outcome is the same isomorphism $M\otimes_RN^\dagger\cong\Hom_R(M,N)^\dagger$. Also note that items $(ii)$ and $(iii)$ of Theorem \ref{lmgeneralization} hold for $d=0$ as well, as soon as $n\geq 1$.

\smallskip

$(ii)$ In the setting of the theorem, the maximal Cohen-Macaulayness of ${\rm Hom}_R(M, N)$ is guaranteed if  $\Ext^j_R(M,N)$ vanishes only for $j=1, \ldots, {\rm max}\{1, d-2\}$; see
Remark \ref{MCMdual}$(i)$. So the key novelty of Theorem \ref{lmgeneralization}$(i)$ is the isomorphism $M\otimes_RN^\dagger\cong\Hom_R(M,N)^\dagger$ (which therefore implies that  $M\otimes_RN^\dagger$ is maximal Cohen-Macaulay). For a version of Theorem \ref{lmgeneralization}$(i)$ if $M$ and $N$ are both maximal Cohen-Macaulay, see \cite[Theorem 5.9]{hujo} for the Gorenstein case (also \cite[Theorem 2.7]{J} for a related result in case $R$ is Cohen-Macaulay but more vanishings are required).
\end{remark}

\begin{definition} \rm Two finite $R$-modules $M$, $N$ are \textit{Tor-independent} if $\Tor^R_i(M,N)=0$ for all $i>0$. \end{definition}

\begin{corollary}\label{CMtensorequivalence}
Let $R$ be a Cohen-Macaulay local ring of dimension $d$ with a canonical module. Let $M$ and $N$ be finite $R$-modules with $N$ maximal Cohen-Macaulay. The following statements hold:
\begin{itemize}
\item [(i)] If $\Ext^j_R(M,N)=0$ for all $j>0$, then $M$ and $N^\dagger$ are Tor-independent;

\item [(ii)] If $M$ and $N^\dagger$ are Tor-independent, then $\Ext^k_R(M,N)=0$ for all $k>d$;

\item [(iii)] If $d>0$ and $\Tor^R_j(M,N^\dagger)=0$ for all $j= 1, \ldots, d$, then $M\otimes_RN^\dagger$ is Cohen-Macaulay of dimension $t$ if and only if $\Ext^i_R(M,N)=0$ for all $i= 0, \ldots, d$ with $i\neq d-t$.
 \end{itemize}
\end{corollary}
\begin{proof}
For $(i)$ (resp. $(ii)$), apply part $(ii)$ (resp. $(iii)$) of Theorem \ref{lmgeneralization} for all $n>0$. To prove $(iii)$, by taking $n=0$ in Theorem \ref{lmgeneralization}$(iii)$ we deduce isomorphisms $$\Ext^i_R(M\otimes N^\dagger,\omega_R)\cong\Ext^i_R(M,N) \quad \mbox{for \,all} \quad i= 0, \ldots, d.$$ The result now follows by local duality.
\end{proof}

\begin{remark}\rm 
Notice that Corollary \ref{CMtensorequivalence}$(iii)$ (which provides a partial converse of Theorem \ref{lmgeneralization}$(i)$) yields the following curious property. If, in the same setting, ${\rm Hom}_R(M, N)\neq 0$ and the tensor product $M\otimes_RN^\dagger$ is Cohen-Macaulay, then in fact it must be maximal Cohen-Macaulay. Another consequence in the same direction is given below.
\end{remark}

\begin{corollary}\label{MCMtensorequivalence}
Let $R$ be a Cohen-Macaulay local ring of dimension $d\geq 1$ with a canonical module. Let $M$ and $N$ be finite $R$-modules with $N$ maximal Cohen-Macaulay. If $\Tor_j^R(M,N^\dagger)=0$ for all $j= 1, \ldots, d$, then the following statements are equivalent:
\begin{itemize}
    \item [(i)] $M\otimes_RN^\dagger$ is maximal Cohen-Macaulay;
    \item [(ii)] $\Ext^i_R(M,N)=0$ for all $i= 1, \ldots, d$.
\end{itemize}
\end{corollary}

\begin{proof}
Theorem \ref{lmgeneralization}$(i)$ (alternatively, \cite[Lemma 3.4(1)]{LM}) guarantees that $(ii)\Rightarrow(i)$, which holds without the Tor vanishing assumption. Corollary \ref{CMtensorequivalence}$(iii)$ yields the implication $(i)\Rightarrow(ii)$.
\end{proof}

Clearly, by letting $N=\omega_R$ (so that $N^{\dagger}\cong R$), this corollary retrieves the standard fact that $M$ is maximal Cohen-Macaulay if and only if $\Ext^i_R(M, \omega_R)=0$ for all $i=1, \ldots, d$.

We close the section by combining Theorem \ref{lmgeneralization} with the so-called (Auslander's) depth formula, which first appeared in \cite{Au} (in the situation where one of the modules involved has finite projective dimension) and has been widely investigated ever since; see, e.g., \cite{AY}, \cite{BJ}, \cite{CJ},  and \cite{HW}. 

\begin{definition}\rm
Two finite $R$-modules $M$ and $N$ are said to {\it satisfy the depth formula} if
\begin{equation} \label{depthform}\depth_R(M\otimes_RN)=\depth_RM+\depth_RN - \depth R.\end{equation}
\end{definition}

The following corollary is a direct application of Theorem \ref{lmgeneralization}$(i)$. Besides its usefulness in Corollary \ref{depthformula2} below, it will also be applied later in the proof of Proposition \ref{gregulargeneralization}.

\begin{corollary}\label{depthformula}
Let $R$ be a Cohen-Macaulay local ring of dimension $d\geq 1$ with a canonical module. Let $M$ and $N$ be finite $R$-modules such that $N$ is maximal Cohen-Macaulay and $\Ext^j_R(M,N)=0$ for all $j= 1, \ldots, d$. If $M$, $N^\dagger$ satisfy the depth formula $($\ref{depthform}$)$, then $M$ is maximal Cohen-Macaulay.
\end{corollary}


\begin{remark}\rm If $R$ is an {\rm AB} ring in the sense of \cite{hujo} (for example,  a complete intersection, i.e., the quotient of a regular local ring by a proper ideal generated by a regular sequence) and $M$ is a finite $R$-module such that  $\Ext^j_R(M,N)=0$ for all $j>0$ and some finite $R$-module $N$, then it follows from \cite[Theorem 3.6]{CJ} that $M$ is maximal Cohen-Macaulay (the converse is clear). 
\end{remark}



Here, we observe the following criteria for maximal Cohen-Macaulay modules over certain (not necessarily {\rm AB}) Cohen-Macaulay rings. In order to state it we need to invoke a few ingredients. First, for a local ring $R$, the Gorenstein dimension (resp. complete intersection dimension) of a finite $R$-module $M$ is written $\gdim_RM$ (resp. $\cidim_RM$), and recall that $R$ is Gorenstein (resp. a complete intersection) if and only if $\gdim_RM<\infty$ (resp. $\cidim_RM<\infty$) for every finite $R$-module $M$; we refer to \cite{AuB} (resp. \cite{AGP}). We have $\gdim_RM\leq\cidim_RM$, with equality if $\cidim_RM<\infty$; more about $\cidim_RM$ will be considered in Subsection \ref{CIdim0}.  Next, given $R$-modules $M, N$ and $i\in {\mathbb Z}$, the notation $\widehat{{\rm Tor}}^R_i(M, N)$ stands as usual for the associated $i$th Tate homology module (for the definition and properties, see \cite[Remark 7.4]{avra-M}, \cite[Section 2]{CJ0}, and \cite{Iacob}). Finally, recall that a local ring $(R, {\mathfrak m})$ is (at most) an isolated singularity if $R_{\mathfrak p}$ is regular for every prime ${\mathfrak p}\neq {\mathfrak m}$. 


\begin{corollary}\label{depthformula2}
Let $R$ be a Cohen-Macaulay local ring of dimension $d\geq 1$ with a canonical module, and let $M$, $N$ be finite $R$-modules with $N$ maximal Cohen-Macaulay. Suppose any one of the following sets of hypotheses:
\begin{itemize}
    \item [(i)] $R$ is an isolated singularity, $\cidim_RM<\infty$ $($or $\cidim_RN^{\dagger}<\infty$$)$, $\Ext^j_R(M,N)=0$ for all $j=1, \ldots, d$, and ${\rm Tor}^R_i(M, N^{\dagger})=0$ for all $i\gg 0$;
    \item [(ii)] $\gdim_RM<\infty$, $\Ext^j_R(M,N)=0$ for all $j>0$, and $\widehat{{\rm Tor}}^R_i(M, N^{\dagger})=0$ for all $i\in {\mathbb Z}$.
    \end{itemize}
Then, $M$ is maximal Cohen-Macaulay. Moreover, if $(i)$  holds then $\Ext^j_R(M,N)=0$ for all $j>0$.
\end{corollary} 
\begin{proof} Suppose $(i)$ holds. Since the module $\Ext^j_R(M,N)$ vanishes for all $j=1, \ldots, d$, we get that $M\otimes_RN^\dagger$ is maximal Cohen-Macaulay (hence torsionfree) by the implication $(ii)\Rightarrow(i)$ of Corollary \ref{MCMtensorequivalence}. Also note $N^{\dagger}$ is maximal Cohen-Macaulay. Now we are in a position to apply \cite[Theorem 2.10]{AY}, which ensures that 
$M$, $N^\dagger$ are Tor-independent. By \cite[Theorem 2.5]{AY}, the modules $M$, $N^\dagger$ satisfy the depth formula $($\ref{depthform}$)$. Therefore, $M$ is maximal Cohen-Macaulay by Corollary \ref{depthformula}.
Notice that Corollary \ref{CMtensorequivalence}$(ii)$ implies $\Ext^k_R(M,N)=0$ for all $k>d$. It follows that $\Ext^j_R(M,N)=0$ for all $j>0$.

Finally, let us assume $(ii)$. Using \cite[Theorem 2.3]{CJ} we obtain that the so-called {\it derived depth formula} (stated more generally for complexes) holds for the modules $M$ and $N^{\dagger}$. On the other hand, by Corollary \ref{CMtensorequivalence}$(i)$, these modules  are Tor-independent. Now, as explained in \cite[Introduction, p.\,465]{CJ}, such properties validate the usual depth formula for $M$ and $N^{\dagger}$.  By Corollary \ref{depthformula} once again, we get that $M$ is maximal Cohen-Macaulay.\end{proof}






\begin{question}\label{ques-MCM}\rm Let $R$ be a complete intersection local ring of dimension $d>0$, and let $M$ be a finite $R$-module. Suppose $\Ext^j_R(M,N)=0$ for all $j=1, \ldots, d$ and some maximal Cohen-Macaulay $R$-module $N$. Is it true that $M$ must be maximal Cohen-Macaulay? If not, what if in addition $R$ is an isolated singularity?
\end{question}


\begin{remark}\rm $(i)$ It is clear that the key issue behind Question \ref{ques-MCM} is whether, at least for complete intersections, Corollary \ref{CMtensorequivalence}$(i)$ remains valid if $\Ext^j_R(M,N)$ is required to vanish only for $j=1, \ldots, d$. 




\smallskip

$(ii)$ By the well-known Ischebeck's theorem (see \cite[Exercise 3.1.24]{BH}), Question \ref{ques-MCM} is easily seen to admit an affirmative answer if $R$ is only required to be Cohen-Macaulay and ${\rm id}_RN<\infty$.



\end{remark}

A closely related question can also be raised concerning Corollary \ref{depthformula2}$(ii)$, as follows.

\begin{question}\label{ques-Gor-MCM}\rm Let $R$ be a Gorenstein local ring of dimension $d\geq 1$ and let $M$, $N$ be finite $R$-modules, with $N$ maximal Cohen-Macaulay. Suppose $\Ext^j_R(M,N)=0$ for all $j=1, \ldots, d$ and $\widehat{{\rm Tor}}^R_i(M, N^{*})=0$ for all $i\in {\mathbb Z}$. Is it true that $M$ must be maximal Cohen-Macaulay?
\end{question}


\section{Freeness criteria over Cohen-Macaulay local rings }\label{freenesssec}

In this part we present several criteria for the freeness of finite modules over Cohen-Macaulay local rings, in connection to a variety of tools. For such reasons, and in order to make the section more organized, we have divided it into three subsections.

\subsection{Test modules and G-regular rings}\label{testmodulessec} We begin by recalling a central concept in this subsection.

\begin{definition}\rm (\cite[Definition 1.1]{CDT}) A finite $R$-module $M$ is said to be a \emph{test} module if all finite $R$-modules $N$ with $\Tor_i^R(M,N)=0$ for all $i\gg0$ satisfy $\pd_RN<\infty$.
\end{definition}


\begin{lemma}$($\cite[Proposition 3.6]{CDT}$)$\label{test-ep}
Let $R$ be a Cohen-Macaulay local ring with  canonical module. For each maximal Cohen-Macaulay $R$-module $M$, the following statements are equivalent:
\begin{itemize}
    \item [(i)] $M$ is a test module;
    \item [(ii)] All finite $R$-modules $N$ with $\Ext^j_R(N,M^\dagger)=0$ for all $j\gg 0$ satisfy $\pd_RN<\infty$.
    

\end{itemize}
\end{lemma}


Next, a remarkable property of the Gorenstein dimension is that $\gdim_RM\leq\pd_RM$ for any finite $R$-module $M$, with equality if $\pd_RM<\infty$ (see \cite{AuB}). Based on this fact, the following notion appears naturally.

\begin{definition}\rm \label{G-reg-def} (\cite{T}) The local ring $R$ is said to be \emph{$G$-regular} if $\gdim_RM=\pd_RM$ for any finite $R$-module $M$.
\end{definition}

In fact the definition was originally given in an apparently different way: $R$ is G-regular if and only if every totally reflexive $R$-module (i.e., a reflexive finite $R$-module $M$ with ${\rm Ext}_R^i(M, R)={\rm Ext}_R^i(M^*, R)=0$ for all $i>0$, or equivalently, $\gdim_RM=0$) is necessarily free. But according to \cite[Proposition 1.8(2)]{T} this is equivalent to Definition \ref{G-reg-def}. Note that regular local rings are trivial examples of G-regular rings. It is also well-known that any Golod local ring that is not a hypersurface (e.g., a Cohen-Macaulay non-Gorenstein local ring with minimal multiplicity) is G-regular.

Before proving our main result in this subsection, we invoke a helpful fact which reconnects us to the aforementioned depth formula.

\begin{lemma}\label{pddepthformula}$($\cite[Theorem 1.2]{Au}$)$
Let $R$ be a local ring. If $M$ and $N$ are Tor-independent $R$-modules with $\pd_RM<\infty$, then $M$ and $N$ satisfy the depth formula $($\ref{depthform}$)$.
\end{lemma}

\begin{proposition}\label{gregulargeneralization}
Suppose $R$ is a Cohen-Macaulay local ring with a canonical module, and let $M$ and $N$ be finite $R$-modules with $N$ maximal Cohen-Macaulay. Assume that the following assertions hold:
\begin{itemize}
    \item [(i)] $\Ext^j_R(M,N)=0$ for all $j>0$;
    \item [(ii)] $N^\dagger$ is a test module.
\end{itemize}
Then, $M$ is free. In particular, if $M=N$ then $R$ is $G$-regular.
\end{proposition}
\begin{proof}
By Theorem \ref{lmgeneralization}$(i)$,$(ii)$, we get that $M\otimes_R N^\dagger$ is maximal Cohen-Macaulay and $\Tor^R_i(M,N^\dagger)=0$ for all $i>0$. Since $N^\dagger$ is a test module, we must have $\pd_RM<\infty$. Thus, Lemma \ref{pddepthformula} yields that $M$ and $N^\dagger$ satisfy the depth formula $($\ref{depthform}$)$, and consequently Corollary \ref{depthformula} ensures that $M$ is maximal Cohen-Macaulay, whence free.

Now, let $N'$ be a finite $R$-module with $\gdim_RN'<\infty$. 
Thus, as any totally reflexive $R$-module $T$
satisfies $\Ext^k_R(T,R)=0$ for all $k>0$, we must have $\Ext^j_R(N',R)=0$ whenever $j\gg0$. Because $M$ is free, we obtain $\Ext^j_R(N',M)=0$ for all $j\gg 0$. Since $({M^\dagger})^{\dagger}\cong M$ and by assumption $M=N$, we can apply Lemma \ref{test-ep} in order to get $\pd_RN'<\infty$. Therefore, $\gdim_RN'=\pd_RN'$. This shows that $R$ is G-regular in the case $M=N$.
\end{proof}

As a byproduct, by taking $M=N=R$ in our Proposition \ref{gregulargeneralization}, we immediately recover the following interesting result from \cite{CDT}.

\begin{corollary} {\rm (}\cite[Corollary 3.7]{CDT}{\rm )}
Let $R$ be a Cohen-Macaulay local ring with canonical module $\omega_R$. If $\omega_R$ is a test module, then $R$ is $G$-regular. 
\end{corollary}



\subsection{Complete intersection dimension zero}\label{CIdim0} Recall that $\gdim_RM\leq\cidim_RM\leq \pd_RM$, with equalities whenever $\pd_RM<\infty$, and \begin{equation}\label{CI-dim-form}\gdim_RM=\cidim_RM={\rm depth}\,R - {\rm depth}\,M\end{equation} if $\cidim_RM<\infty$.  See \cite{AGP} for details.


\begin{lemma}$($\cite[Theorem 4.2]{AB}$)$\label{cidim}
Let $R$ be a local ring and $M$ be a finite $R$-module with $\cidim_RM<\infty$. Then, $\pd_RM<\infty$ if and only if $\Ext^e_R(M,M)=0$ for some even integer $e\geq2$.
\end{lemma}

Inspired by the notion of G-regular ring, we may ask the following natural question.

\begin{question}\label{ques-CI-reg}\rm When does the property $\cidim_RM=0$ force $M$ to be free? (In analogy to the concept of G-regular ring, it would be natural to coin the term CI-{\it regular} for a local ring $R$ with the property that every finite $R$-module with complete intersection dimension zero is necessarily free; notice, therefore, that every G-regular ring is CI-regular.)
\end{question}

Clearly, the answer is not always affirmative, since there are plenty of examples of (non-regular) complete intersection rings admitting non-free maximal Cohen-Macaulay modules. On the other hand, we have the following two results.

\begin{proposition}\label{cidimtor}
Let $R$ be a Cohen-Macaulay local ring of dimension $d$ with a canonical module. Suppose $M$ is a finite $R$-module with $\cidim_RM=0$ and  $$\Tor^R_j(M,M^\dagger)=0 \quad \mbox{for \,all} \quad j= 1, \ldots, t,$$ where either $t=d+1$ if $d$ is odd, or $t=d+2$. Then, $M$ is free.

\end{proposition}
\begin{proof} Using (\ref{CI-dim-form}), we have \, $0= \cidim_RM ={\rm depth}\,R - {\rm depth}\,M = d - {\rm depth}\,M$, so $M$ is maximal Cohen-Macaulay. Then
Theorem \ref{lmgeneralization}$(iii)$ ensures that $\Ext^k_R(M,M)=0$ whenever $d+1\leq k\leq t$. Now the result follows by Lemma \ref{cidim}.
\end{proof}



A finite module over a local ring $(R, \mathfrak{m})$ is said to be a \textit{vector bundle} if it is locally free on the punctured spectrum ${\rm Spec}\,R\setminus \{\mathfrak{m}\}$ of $R$.

\begin{proposition}\label{exttor}
Let $R$ be a Cohen-Macaulay local ring of dimension $d$ with a canonical module. Let $n\geq 0$ be an integer such that $d+n\geq 1$ is odd. If $M$ is a vector bundle with $\cidim_RM=0$, satisfying $$\Ext^j_R(M,M)=0 \quad \mbox{for \,all} \quad j= 1, \ldots, d+n$$ and $\Tor^R_{n+1}(M,M^\dagger)=0$, then $M$ is free.
\end{proposition}

\begin{proof} As $M$ is a vector bundle, ${\rm length}\Ext^{d+n+1}_R(M,M)<\infty$. Using local duality, we must have $\Ext^i_R(\Ext^{d+n+1}_R(M,M),\omega_R)=0$ for all $i\neq d$.
Moreover, as in the proof of Proposition \ref{cidimtor}, the $R$-module $M$ is maximal Cohen-Macaulay. Applying Theorem \ref{lmgeneralization}$(ii)$, we obtain $$\Ext^d_R(\Ext^{d+n+1}_R(M,M),\omega_R)\cong\Tor_{n+1}^R(M,M^\dagger)=0,$$ which thus forces $\Ext^{d+n+1}_R(M,M)=0$. Now the statement follows by Lemma \ref{cidim}.
\end{proof}



\subsection{Main freeness criteria} First we introduce some ingredients. Among the results in this section, we shall present a freeness criterion that generalizes the main one established in \cite{ACST}.


Let $R$ be a local ring. Given a finite $R$-module $M$, we can consider a minimal $R$-free presentation $\xymatrix@=1em{G\ar[r] & F\ar[r] & M\ar[r] & 0}$. The ({\it Auslander}) {\it transpose} of $M$, denoted $\Tr M$, is defined as the cokernel of the dual map $\xymatrix@=1em{F^*\ar[r] & G^*}$. The transpose is uniquely determined up to isomorphism.


\begin{lemma}$($\cite[Proposition 3.2 (i)]{KOT}$)$ \label{trlemma}
Let $R$ be a local ring and $M, N$ be finite $R$-modules. If $N\neq0$ and $\Tor_1^R(\Tr M,M\otimes_R N)=0$, then $M$ is free.
\end{lemma}

\begin{lemma}$($\cite[Lemma 2.2]{Y}$)$ \label{tor0}
Let $R$ be a local ring and let $M,N$ be finite $R$-modules such that $\pd_RM<\infty$ and $N$ is maximal Cohen-Macaulay. Then, $M$ and $N$ are Tor-independent.
\end{lemma}

For a ring $R$ and a non-negative integer $n$, we set $\X^n(R)=\{\mathfrak{p}\in\Spec\,R \mid \height\mathfrak{p}\leq n\}$. For instance, saying that a finite $R$-module $M$ is locally of finite projective dimension on $\X^n(R)$ means that $\pd_{R_\mathfrak{p}}M_\mathfrak{p}<\infty$ for all $\mathfrak{p}\in\X^n(R)$. Hence, $M$ is locally free on $\X^n(R)$ if $M_\mathfrak{p}$ is free over $R_\mathfrak{p}$ for all $\mathfrak{p}\in\X^n(R)$. Note that, if $R$ is local, $M$ being locally free on $\X^{\dim R-1}(R)$ means exactly that $M$ is a vector bundle.

\begin{lemma}$($\cite[2.4]{ACST}$)$ \label{takahashidual}
Let $R$ be a Cohen-Macaulay local ring of dimension $d\geq1$ with canonical module $\omega_R$, and let $M, N$ be finite $R$-modules. If $M$ is locally of finite projective dimension on $\X^{d-1}(R)$ and $N$ is maximal Cohen-Macaulay, then  $$\Ext^{d+j}_R(M,N^\dagger)\cong\Ext^d_R(\Tor_j^R(M,N),\omega_R) \quad \mbox{for \,all} \quad j\geq 1.$$
\end{lemma}


Given a non-negative integer $n$, a finite $R$-module $M$ is said to satisfy the \textit{Serre condition} $(S_n)$ if $\depth_{R_\mathfrak{p}}M_{\mathfrak{p}}\geq\min\{n, \height\mathfrak{p}\}$ for all $\mathfrak{p}\in\Supp M$.

The theorem below, when applied with $N=\omega_R$, provides a generalization of \cite[Theorem 1.4]{ACST} (the main result therein) different from that obtained in \cite[Corollary 3.9]{STAK}. Although the latter just requires $M$ to satisfy $(S_1)$ (whereas we need $(S_2)$), it also needs $M^*$ to be maximal Cohen-Macaulay (equivalently, $\Ext^j_R(M^*, \omega_R)=0$ for all $j=1, \ldots, d$), while
our result makes use of any maximal Cohen-Macaulay $R$-module $N$.

\begin{theorem}\label{acstgeneralization}
Let $R$ be a Cohen-Macaulay local ring of dimension $d\geq2$, and let $M, N$ be finite $R$-modules. Assume $n$ is an integer such that $1\leq n\leq d-1$. Suppose the following conditions:  \begin{itemize}
    \item [(i)] $M$ is locally of finite projective dimension on $\X^n(R)$;
    \item [(ii)] $M$ satisfies $(S_2)$ and $N$ is maximal Cohen-Macaulay;
    \item [(iii)] $\Ext^j_R(M^*,N)=0$ for all $j=1, \ldots, d$;
    \item [(iv)] $\Ext^i_R(M,\Hom_R(M^*,N))=0$ for all $i=n, \ldots, d-1$.
\end{itemize}
Then, $M$ is free.
\end{theorem}
\begin{proof}
Recall $M$ satisfies $(S_2)$ and notice that $\pd_{R_\mathfrak{q}}M_{\mathfrak{q}}<\infty$ for all $\mathfrak{q}\in\X^1(R)$. Therefore, $M$ is locally free (hence locally reflexive) on $\X^1(R)$, and then by \cite[Proposition 1.4.1(b)]{BH} we deduce that $M$ is reflexive. Hence, $M$ is isomorphic to the second syzygy module of $\Tr M^*$, which implies $\pd_{R_\mathfrak{p}}\Tr M^*_\mathfrak{p}<\infty$ whenever $\mathfrak{p}\in\X^n(R)$. Note $R$ can be assumed to be ${\mathfrak m}$-adically complete, so $R$ possesses a canonical module. As $N$ is maximal Cohen-Macaulay and in virtue of the hypothesis $(iii)$, we can apply Theorem \ref{lmgeneralization}$(i)$ to obtain that $M^*\otimes_RN^\dagger$ is maximal Cohen-Macaulay and $M^*\otimes_RN^\dagger\cong\Hom_R(M^*,N)^\dagger$. Now Lemma \ref{tor0} gives $\Tor^R_1(\Tr M^*, M^*\otimes_RN^\dagger)_\mathfrak{p}=0$ for all $\mathfrak{p}\in\X^n(R)$, and Lemma \ref{trlemma} in turn ensures that $M^*_\mathfrak{p}$ is free over $R_\mathfrak{p}$ for all $\mathfrak{p}\in\X^n(R)$; since $M$ is reflexive, $M_\mathfrak{p}$ must be a free $R_\mathfrak{p}$-module as well.

Now, pick $\mathfrak{p}\in\Spec\,R$ with $\height\mathfrak{p}=h>n$. After localizing at $\mathfrak{p}$, we can assume that $R$ is a Cohen-Macaulay local ring of dimension $h$, and that the $R$-module $M$ is reflexive and (by induction on $h$) a vector bundle, satisfying in addition $\Ext^{h-1}_R(M,\Hom_R(M^*,N))=0$. Proceeding as above, we get $${\rm length}\Tor_1^R(\Tr M^*,M^*\otimes_RN^\dagger) < \infty.$$ By local duality (below, $H_{\mathfrak m}^0$ denotes the 0-th local cohomology functor) and Lemma \ref{takahashidual}, we have $$\begin{array}{lll}\Tor_1^R(\Tr M^*, M^*\otimes_RN^\dagger)&\cong H^0_\mathfrak{m}(\Tor_1^R(\Tr M^*, M^*\otimes_RN^\dagger))\\ & \cong\Ext^h_R(\Tor_1^R(\Tr M^*, M^*\otimes_RN^\dagger),\omega_R)^\vee\\& \cong\Ext^{h+1}_R(\Tr M^*, (M^*\otimes_RN^\dagger)^\dagger)^\vee\\& \cong \Ext^{h-1}_R(M,\Hom_R(M^*,N))^\vee=0.\end{array}$$ Finally, by Lemma \ref{trlemma}, the $R$-module $M^*$ is necessarily free and hence so is $M$.
\end{proof}

We record the main result of \cite{ACST} as a corollary.

\begin{corollary}$($\cite[Theorem 1.4]{ACST}$)$\label{1.4acst}
Let $R$ be a Cohen-Macaulay local ring of dimension $d\geq2$ with a canonical module, and let $M$ be a finite $R$-module. Assume $n$ is an integer such that $1\leq n\leq d-1$. Suppose the following conditions:
\begin{itemize}
    \item [(i)] $M$ is locally of finite projective dimension on $\X^n(R)$;
    \item [(ii)] $M$ satisfies $(S_2)$ and $M^*$ is maximal Cohen-Macaulay;
    \item [(iii)] $\Ext^i_R(M,(M^*)^\dagger)=0$ for all $i=n, \ldots, d-1$.
\end{itemize} Then, $M$ is free.
\end{corollary}

\begin{remark}\label{MCMdual}\rm $(i)$ For a Cohen-Macaulay local ring $R$ of dimension $d$ and a finite $R$-module $M$, it is of interest (e.g., in view of Corollary \ref{1.4acst} above) to investigate when $M^*$ is maximal Cohen-Macaulay. In this regard, there is the following general result which follows from \cite[Lemma, p.\,2763]{JD}. If $R$ is a local ring and $N$ is a maximal Cohen-Macaulay $R$-module such that
$$\Ext^i_R(M, N)=0 \quad \mbox{for \,all} \quad i=1, \ldots, {\rm max}\{1, d-2\},$$ then ${\rm Hom}_R(M, N)$ is maximal Cohen-Macaulay.

\smallskip

$(ii)$ According to \cite[Corollary 3.4]{Dao-Se}, if $R$ is a Cohen-Macaulay local ring with a canonical module and $M$ is a finite $R$-module satisfyng $(S_2)$ and such that $M^{\dagger}$ is maximal Cohen-Macaulay, then $M$ is maximal Cohen-Macaulay as well. Along with part $(i)$
above, we conclude that if $M$ satisfies $(S_2)$ and $\Ext^i_R(M, \omega_R)=0$ for all $i=1, \ldots, {\rm max}\{1, d-2\}$, then $M$ is maximal Cohen-Macaulay. This fact seems to suggest the question below.
\end{remark}
 
\begin{question}\rm Let $R$ be a Cohen-Macaulay local ring with canonical module $\omega_R$ and dimension $d$. Let $M$ be a finite $R$-module satisfying $(S_k)$ with either $k=1$ or $3\leq k<d$, and such that $\Ext^i_R(M, \omega_R)=0$ for all $i=1, \ldots, {\rm max}\{1, d-k\}$. It is true that $M$ is maximal Cohen-Macaulay?
\end{question}

Next, we are able to use our Theorem \ref{acstgeneralization} to recover yet another result from the literature (precisely, from \cite{A}) by taking $n=d-1$ and $N=R$, with $R$ Gorenstein.

\begin{corollary}$($\cite[Corollary 10]{A}$)$\label{arayacorollary}
Let $R$ be a Gorenstein local ring of dimension $d\geq2$, and let $M$ be a maximal Cohen-Macaulay vector bundle. If $\Ext^{d-1}_R(M,M)=0$, then $M$ is free.
\end{corollary}



In the sequel, we provide further criteria for the freeness of modules over Cohen-Macaulay rings.

\begin{proposition}\label{freenesstr}  Let $R$ be a Cohen-Macaulay local ring of dimension $d$, and let $M,N$ be finite $R$-modules with $N$ maximal Cohen-Macaulay. Suppose the following conditions:
\begin{itemize}
\item [(i)] $\Ext^i_R(M,N)=0$ for all $i=1, \ldots, d$ if $d\geq 1$, or for all $i>0$ if $d=0$;
\item [(ii)] $\Ext^j_R(\Tr M,\Hom_R(M,N))=0$ for all $j=1, \ldots, d+1$.
\end{itemize} Then, $M$ is free.
\end{proposition}
\begin{proof}
We can assume that $R$ is complete, hence it admits a canonical module. Because of the hypothesis $(i)$ and Theorem \ref{lmgeneralization}$(i)$ (also Remark \ref{main-Art}$(i)$ in the Artinian case), we obtain that the $R$-modules $\Hom_R(M,N)$ and $M\otimes_RN^\dagger\cong\Hom_R(M,N)^\dagger$ are maximal Cohen-Macaulay. On the other hand, due to $(ii)$ and Theorem \ref{lmgeneralization}$(ii)$ (and again Remark \ref{main-Art}$(i)$ if $d=0$), we have $\Tor^R_1(\Tr M, M\otimes_RN^\dagger)=0$. The result now follows by Lemma \ref{trlemma}.
\end{proof}

Applying this proposition with $N=\omega_R$, we obtain the following result.

\begin{corollary}\label{trfreenesscriterion}
Let $R$ be a Cohen-Macaulay local ring of dimension $d$ with a canonical module, and let  $M$ be a maximal Cohen-Macaulay $R$-module. Then, $M$ is free if $$\Ext^j_R(\Tr M,M^\dagger)=0 \quad \mbox{for \,all} \quad j=1, \ldots, d+1.$$ \end{corollary}



Next, we apply Corollary \ref{trfreenesscriterion} in the particular cases $d=0$ and $d=1$ in order to get, respectively, the following statements. 


\begin{corollary}
Let $R$ be an Artinian local ring with a canonical module. If $M$ is a finite $R$-module such that $\Ext^1_R(\Tr M,M^\dagger)=0$, then $M$ is free.
\end{corollary}


\begin{corollary}\label{HIWconjectureext12}
Let $R$ be a one-dimensional Cohen-Macaulay local ring with a canonical module, and let $M$ be a finite torsionfree $R$-module. Then, $M$ is free if $$\Ext^1_R(\Tr M,M^\dagger)=\Ext^2_R(\Tr M,M^\dagger)=0.$$
\end{corollary}

Before providing another consequence of Proposition \ref{freenesstr} (along with Theorem \ref{lmgeneralization}), we need to invoke a couple of auxiliary lemmas. If $U$ is a finite $R$-module, we can consider its \emph{non-free locus}  $\NF(U)=\{\mathfrak{p}\in\Spec\,R \mid U_\mathfrak{p} \ \mbox{is not} \ R_\mathfrak{p}\mbox{-free}\}$.

\begin{lemma}$($\cite[Theorem 2.8 (2)]{STAK}$)$ \label{nonfreelemma}
Let $R$ be a Cohen-Macaulay local ring of dimension $d\geq1$ with a canonical module. Let $M$ and $N$ be finite $R$-modules with $M$ locally totally reflexive on the punctured spectrum of $R$ and $N$  maximal Cohen-Macaulay, such that $\NF(M)\cap \NF(N)\subseteq\{\mathfrak{m}\}$. Then, there exist isomorphisms $$\Ext^j_R(M,N)^\vee\cong\Ext^{d+1-j}_R(\Tr M,N^\dagger) \quad \mbox{for \,all} \quad j= 1, \ldots, d.$$ 
\end{lemma}

\begin{lemma}\label{freetensor} Let $R$ be a local ring and $T, U$ be non-zero finite $R$-modules such that $T\otimes_RU$ is free. Then, $T$ and $U$ are free.
\end{lemma} 
\begin{proof} This is an elementary fact whose proof we supply for convenience. The condition of $T$ and $U$ being non-zero is equivalent to $T\otimes_RU\neq 0$. Now pick an $R$-free presentation $\xymatrix@=1em{F_1\ar[r] & F_0\ar[r] & T\ar[r] & 0}$ of the module $T$. Tensoring it with $U$, and using that the module $G=T\otimes_RU$ is free, the resulting map splits so that $G$ is a direct summand of $F_0\otimes_RU$. Tensoring such a decomposition with $T$, we obtain that $G\otimes_RT$ is a direct summand of the module $F_0\otimes_RU\otimes_RT=F_0\otimes_RG$, which is free, i.e., $G\otimes_RT$ is projective, hence free because $R$ is local. This is equivalent to $T$ being free. By a completely analogous argument, $U$ must be free as well.
\end{proof}

\begin{corollary}\label{NFlocus}
Let $R$ be a Cohen-Macaulay local ring of dimension $d\geq1$ with a canonical module. Let $M$ and $N$ be finite $R$-modules such that $M$ is a vector bundle and $N$ is maximal Cohen-Macaulay with ${\rm Supp}_RN={\rm Spec}\,R$. Suppose the following conditions:
\begin{itemize}
    \item [(i)] $\Ext^i_R(M,N)=0$ for all $i= 1, \ldots, d$;
    \item [(ii)] $\Ext^j_R(M,M\otimes_RN^\dagger)=0$ for all $j= 1, \ldots, d$;
    \item [(iii)] $\Ext^{d+1}_R(\Tr M, \Hom_R(M,N))=0$\, $($e.g., if\, ${\rm pd}_RM^*<\infty$\, or\, ${\rm id}_R\Hom_R(M,N)<\infty$$)$.
\end{itemize} Then, $M$ is free.
\end{corollary}
\begin{proof} First, we claim that if ${\mathfrak p}\in {\rm Spec}\,R$ is such that the $R_{\mathfrak p}$-module $M_{\mathfrak p}\otimes_{R_{\mathfrak p}}{N_{\mathfrak p}^\dagger}$ is free, then $M_{\mathfrak p}$ is free. We may assume $M_{\mathfrak p}\neq 0$. Since $N\cong (N^{\dagger})^{\dagger}$ and ${\rm Supp}_RN={\rm Spec}\,R$, we must have ${N_{\mathfrak p}^\dagger}\neq 0$. By Lemma
\ref{freetensor}, $M_{\mathfrak p}$ is free.
This proves that
$\NF(M)\subseteq \NF(M\otimes_RN^\dagger)$. Hence, $$\NF(M)\cap \NF(M\otimes_RN^\dagger)=\NF(M)\subseteq\{\mathfrak{m}\}.$$ By condition $(i)$ and Theorem \ref{lmgeneralization}$(i)$, the $R$-modules $M\otimes_RN^\dagger\cong\Hom_R(M,N)^\dagger$ and $\Hom_R(M,N)$ are maximal Cohen-Macaulay. Now we are in a position to apply Lemma \ref{nonfreelemma} to get $$\Ext^j_R(M,M\otimes_RN^\dagger)^\vee\cong\Ext^{d+1-j}_R(\Tr M,\Hom_R(M,N)) \quad \mbox{for \,all} \quad j=1, \ldots, d.$$ Therefore, using $(ii)$ we deduce that $$\Ext^k_R(\Tr M,\Hom_R(M,N))=0 \quad \mbox{for \,all} \quad k=1, \ldots, d.$$ The result now follows by Proposition \ref{freenesstr}.
\end{proof}

\begin{example}\rm Not every maximal Cohen-Macaulay module $N$ over a Cohen-Macaulay (or even a hypersurface) local ring $(R, {\mathfrak m})$
must satisfy ${\rm Supp}_RN={\rm Spec}\,R$. Perhaps the easiest example is $R={\bf k}[[x, y]]/(xy)$, where ${\bf k}$ is a field, and $N=R/yR$, which is maximal Cohen-Macaulay. Now let ${\mathfrak p}=xR\in {\rm Spec}\,R$. Then $N_{{\mathfrak p}}=0$. Precisely, ${\rm Supp}_RN=\{{\mathfrak m}, yR\}$.
\end{example}

To close the section, by applying Corollary \ref{NFlocus} with $N=\omega_R$ (which satisfies ${\rm Supp}_R\omega_R={\rm Spec}\,R$, since ${\rm End}_R(\omega_R)\cong R$), we get the following freeness characterization.

\begin{corollary}\label{mcmvectorbundle}
Let $R$ be a Cohen-Macaulay local ring of dimension $d\geq1$ with a canonical module. Let $M$ be a maximal Cohen-Macaulay $R$-module which is a vector bundle. Suppose the following conditions:
\begin{itemize}
    \item [(i)] $\Ext^i_R(M, M)=0$ for all $i= 1, \ldots, d$;
    \item [(ii)] $\Ext^{d+1}_R(\Tr M,M^\dagger)=0$\, $($e.g., if\, ${\rm pd}_RM^*<\infty$\, or\, ${\rm id}_RM ^{\dagger}<\infty$$)$.
\end{itemize} Then, $M$ is free.
\end{corollary}



\section{Applications to the Auslander-Reiten conjecture and related problems}\label{conjecturessec}

In this section we present our main applications regarding the classical Auslander-Reiten conjecture and important related problems such as the Tachikawa and Huneke-Wiegand conjectures.




Following standard notation, for a local ring $(R, \mathfrak{m})$ we write $\widehat{R}$ for its ($\mathfrak{m}$-adic) completion.


\begin{definition}$($\cite[Definition 2.1, Definition 2.7]{DKT}$)$ \rm Let $(R, \mathfrak{m})$ be a local ring and $I$ be an ideal of $R$. Then $I$ is a {\it Burch ideal} if $\mathfrak{m}I\neq\mathfrak{m}(I:_R\mathfrak{m})$ (notice that this forces ${\rm depth}\,R/I=0$). Now let ${\rm depth}\,R=t$. Then $R$ is a {\it Burch ring} if there exist a maximal $\widehat{R}$-sequence $\underbar{x}=x_1,...,x_t$ in $\widehat{R}$, a regular local ring $S$ and a Burch ideal $I$ of $S$ such that $\widehat{R}/(\underbar{x})\cong S/I$.
\end{definition}

In particular, a Burch ring of depth zero (e.g., an Artinian Burch ring) is just a local ring whose completion is isomorphic to the quotient of a regular local ring by a Burch ideal; for example, from \cite{DKT} it is known that if $S=k[[x, y]]$ (where $x, y$ are formal indeterminates over a field $k$) then the rings $S/(x^2, xy^2, y^4)$,  $S/(x^4, y^4, x^3y, xy^3)$ and $S/(x^a, xy, y^b)$ (for all integers $a, b\geq 1$) are Burch. It is also worth recalling that in order to settle the Auslander-Reiten conjecture in the Cohen-Macaulay case it suffices to consider Artinian rings (see \cite[Introduction]{HSV}).

\begin{lemma}$($\cite[Theorem 1.4]{DKT}$)$\label{burchringfinitepdlemma}
Let $R$ be a Burch ring of depth $t$, and let $M$, $N$ be finite $R$-modules. Suppose there exists an integer $k \geq\max\{3, t+1\}$ such that $\Tor^R_j(M,N)=0$ for all $j=k  +t, \ldots, k +2t+1$. Then either $\pd_RM<\infty$ or $\pd_RN<\infty$.
\end{lemma}


\begin{proposition}\label{burchfreeness}
Let $R$ be a Cohen-Macaulay Burch ring with canonical module $\omega_R$. Let $M$ and $N$ be finite $R$-modules with $N$ maximal Cohen-Macaulay. If $\Ext^j_R(M,N)=0$ for all $j>0$, then either $M$ is free or $N$ is isomorphic to finitely many copies of $\omega_R$.
\end{proposition}
\begin{proof}
By Corollary \ref{CMtensorequivalence}$(i)$, we obtain that the modules $M$ and $N^\dagger$ are Tor-independent. Thus Lemma \ref{burchringfinitepdlemma} applies to ensure that either $\pd_RM<\infty$ or $\pd_RN^\dagger<\infty$. If $\pd_RM<\infty$, then $$\pd_RM=\sup\{j\geq 0  \mid  \Ext^j_R(M,N)\neq 0\}=0.$$ Finally, if $\pd_RN^\dagger<\infty$ then $N^\dagger$ must be free as it is maximal Cohen-Macaulay. Since $N\cong N^{\dagger\dagger}$, the assertion about the structure of $N$ follows.
\end{proof}



\begin{corollary}\label{arburchcondition}
Let $R$ be a non-Gorenstein Cohen-Macaulay Burch ring. If $M$ is a finite $R$-module such that $\Ext^j_R(M,R)=0$ for all $j>0$, then $M$ is free.
\end{corollary}
\begin{proof} Notice that all the conditions present in Proposition \ref{burchfreeness} are preserved after completion, and so we can suppose $R$ admits a canonical module $\omega_R$. By taking $N=R$, we get that either $M$ or $\omega_R$ is free; but the latter is impossible since $R$ is prevented from being Gorenstein.
\end{proof}

\begin{theorem}\label{arburch}
The Auslander-Reiten conjecture $($\ref{arconjecture}$)$ holds true for Cohen-Macaulay Burch rings.
\end{theorem}
\begin{proof}
Let $R$ be such a ring and let $M$ be a finite $R$-module satisfying the hypotheses of (\ref{arconjecture}). If $R$ is Gorenstein, then $R$ must be a hypersurface by \cite[Proposition 5.1]{DKT}, whence the result as the conjecture is true for complete intersections (see \cite[Theorem 4.3]{AY}). Finally, if $R$ is not Gorenstein, we apply Corollary \ref{arburchcondition}.
\end{proof}

For the same class of rings, we immediately derive the validity of the Tachikawa conjecture.


\begin{corollary} The Tachikawa conjecture $($\ref{tachikawanconjecture}$)$ holds true for Cohen-Macaulay Burch rings.
\end{corollary}

In the same spirit, we suggest the following problem (for which, later in Corollary \ref{Gor-charact-d}, we shall provide an affirmative answer in case $R$ has dimension at least $2$ and is locally Gorenstein in codimension 1).

\begin{question}\rm Let $R$ be a Cohen-Macaulay local ring of dimension $d\geq 1$ with canonical module $\omega_R$. If $\Ext^j_R(\omega_R^*,R)=0$ for all $j>0$ $($or for all $j=1, \ldots, d$$)$, must $R$ be Gorenstein?
\end{question}





Besides Theorem \ref{arburch},
another result toward the Auslander-Reiten conjecture is the following immediate consequence of Corollary \ref{mcmvectorbundle}. 


\begin{corollary}\label{ARisolatedsingularity} The Auslander-Reiten conjecture $($\ref{arconjecture}$)$ is true in case $R$ is Cohen-Macaulay of dimension $d\geq1$ with a canonical module and, in addition, $M$ is a
maximal Cohen-Macaulay vector bundle such that $\Ext^{d+1}_R(\Tr M,M^\dagger)=0$ $($this vanishing holds, e.g., if ${\rm pd}_RM^*<\infty$\, or\, ${\rm id}_RM^{\dagger}<\infty$$)$.
\end{corollary}


Next, we record our contribution toward the Huneke-Wiegand conjecture. 



\begin{corollary}\label{mcmvectorbundle2}
The Huneke-Wiegand conjecture $($\ref{hiwconjecture}$)$ is true in case $R$ is $($possibly neither Gorenstein nor a domain$)$ reduced with a canonical module and, in addition, $\Ext^2_R(\Tr M,M^\dagger)=0$.
\end{corollary}
\begin{proof} Note that $R$, being reduced, is locally regular in codimension 0 and hence $M$ is a (maximal Cohen-Macaulay) vector bundle. Now the result follows by Corollary \ref{mcmvectorbundle} with $d=1$.
\end{proof}




We close the section by illustrating that the hypothesis $\Ext^2_R(\Tr M,M^\dagger)=0$ cannot be removed. 


\begin{example}\rm
Take $R=\mathbf{k}[[x,y]]/(xy)$, where $\mathbf{k}$ is a field, and $M=R/xR$. Then, $R$ is a one-dimensional reduced local ring, $\Tr M=M$ is torsionfree, and $M^\dagger= M^*\cong yR$, so that $$\Ext^1_R(M,M)=0 \quad \mbox{and} \quad \Ext^2_R(\Tr M,M^\dagger)\cong yR/y^2R \neq 0,$$ while $\pd_RM=\infty$.
\end{example}

\section{More applications: Regular, complete intersection, and Gorenstein properties}

In this last section, we apply some of our results to derive new criteria for a local ring $R$ to be regular, or a complete intersection, or Gorenstein. As all the converses will hold trivially, our results are in fact characterizations of such properties. We are inspired by \cite{H-L-0} (although the criteria established there work only locally on the punctured spectrum). Moreover, in our investigation of regular local rings, various connections to certain classical differential conjectures will be detected. 

\subsection{Regular local rings, and some differential conjectures} Let us start by recording the following characterization of when a Gorenstein local ring is regular based on the vanishing of Ext modules.

\begin{corollary}\label{cor-gregulargeneralization}
Suppose $R$ is a Gorenstein local ring and $M$ is a maximal Cohen-Macaulay $R$-module. If $\Ext^j_R(M, M)=0$ for all $j>0$ and $M^*$ is a test module, then $R$ is regular.
\end{corollary}
\begin{proof} Applying Proposition \ref{gregulargeneralization} with $N=M$, we get that $R$ is G-regular. Since in addition $R$ is a Gorenstein local ring, it must be regular by \cite[Proposition 1.8(1)]{T}. The converse is clear since any finite free module over a regular local ring is a test module.
\end{proof}

Now our central goal is to establish (co)homological criteria for the regularity of certain local ${\bf k}$-algebras $R$, where we fix ${\bf k}$ as being a field of characteristic zero, admitting a universally finite differential module $\Omega_{R/{\bf k}}=\Omega_{R}$ (see \cite{K} for the theory). This is the case if, for instance, $R$ is either 
\begin{equation}\label{two-classes}
{\bf k}[\underline{x}]_{\mathfrak q}/I \ \ \ ({\mathfrak q}\in {\rm Spec}\,{\bf k}[\underline{x}])\quad \mbox{or} \quad {\bf k}[[\underline{x}]]/I\end{equation} where $I$ is a proper ideal and $\underline{x}=x_1, \ldots, x_m$ are indeterminates over ${\bf k}$ (see \cite[p.\,147]{HuS}). In the first situation, $\Omega_{R}$ is just the module of K\"ahler differentials of $R$ over ${\bf k}$.

The module $\Omega_{R}$ is finite and carries important properties (which will be freely used in the next proofs) such as the following ones. 

\begin{itemize}
    \item [(a)] For any $R$-module $N$, there is an isomorphism ${\rm Hom}_R(\Omega_{R}, N)\cong {\rm Der}_{\bf k}(R, N)$, where the latter is the module formed by the ${\bf k}$-derivations of $R$ with values in $N$. When $N=R$ we retrieve the classical derivation module ${\rm Der}_{\bf k}(R, R)={\rm Der}_{\bf k}(R)={\rm Der}(R)$ of the ${\bf k}$-algebra $R$. Note $\Omega_R^*\cong {\rm Der}(R)$ and, by taking $N=\omega_R$ (the canonical module of $R$), we get $$\Omega_R^{\dagger}\cong {\rm Der}(R, \omega_R).$$

    \item [(b)] Let $R$ be as in (\ref{two-classes}). Then $\Omega_{R}$ is free if and only if $R$ is regular. For a proof of this in case $R={\bf k}[\underline{x}]_{\mathfrak q}/I$ (resp. $R={\bf k}[[\underline{x}]]/I$), see \cite[Theorem 7.2]{K} (resp. \cite[Theorem 14.1]{K}). Note in particular that ${\rm Der}(R)$ is free if $R$ is regular.
   \end{itemize}
   
Our first results on differential and derivation modules deal with the case where, for every prime ideal ${\mathfrak Q}$ of ${\bf k}[\underline{x}]_{\mathfrak q}$ (resp. ${\bf k}[[\underline{x}]]$) containing $I$ and of height at most a suitable integer $n$, the localization $I_{\mathfrak Q}$ is radical and generated by a regular sequence.


\begin{corollary}\label{best-reg} Let $R$ be as in $($\ref{two-classes}$)$. Assume $R$ is Cohen-Macaulay of dimension $d\geq 2$ and locally a reduced complete intersection in codimension $n$, for an integer $n$ such that $1\leq n\leq d-1$. Assume that $\Omega_{R}$ is reflexive. Suppose any one of the following sets of conditions:  \begin{itemize}
    \item [(i)] $\Ext^j_R({\rm Der}(R), R)=0$ for all $j=1, \ldots, d$, and $\Ext^i_R(\Omega_{R}, \Omega_{R})=0$ for all $i=n, \ldots, d-1$;
    
    
    \item [(ii)] $\Ext^j_R(\Omega_{R}, R)=0$ for all $j=1, \ldots, d-2$ $($if $d\geq 3$$)$, $\Ext^i_R({\rm Der}(R), {\rm Der}(R))=0$ for all $i=1, \ldots, d$, and $\Ext^k_R(\Omega_{R}, {\rm End}_R({\rm Der}(R)))=0$ for all $k=n, \ldots, d-1$.
    
    \end{itemize} Then, $R$ is regular.
\end{corollary}
\begin{proof} As $R$ is locally a reduced complete intersection in codimension $n$, we can localize at any given ${\mathfrak p}={\mathfrak P}/I \in X^n(R)$ the conormal exact sequence of $R$ in order to get a short exact sequence
$$\xymatrix@=1em{0\ar[r] & I_{\mathfrak P}/I^{2}_{\mathfrak P}\ar[r] & F\ar[r] & \Omega_{R_{\mathfrak p}}\ar[r] & 0}$$ where $F$ is a finite free $R_{\mathfrak p}$-module. But, under the present hypotheses, it is well-known that $I_{\mathfrak P}/I^2_{\mathfrak P}$ is a free $R_{\mathfrak p}$-module. Therefore, ${\rm pd}_{R_{\mathfrak p}}\Omega_{R_{\mathfrak p}}\leq 1$ for all  ${\mathfrak p} \in X^n(R)$. Moreover notice that, as in particular $R$ is locally Gorenstein in codimension 1, the condition of $\Omega_{R}$ being reflexive is equivalent to $\Omega_{R}$ satisfying $(S_2)$. If we suppose $(i)$, we can apply Theorem \ref{acstgeneralization} with $M=\Omega_{R}$ and $N=R$ in order to get that $\Omega_{R}$ is free, i.e., $R$ is regular.

Now, suppose $(ii)$ holds. We claim that ${\rm Der}(R)$ is maximal Cohen-Macaulay. This property holds if $d=2$ simply because ${\rm Der}(R)$ is a dual over a Cohen-Macaulay ring. Now assume $d\geq 3$. By Remark \ref{MCMdual}$(i)$, the vanishing of $\Ext^j_R(\Omega_R, R)$ whenever $j=1, \ldots, d-2={\rm max}\{1, d-2\}$ forces the module $\Omega_R^*\cong {\rm Der}(R)$ to be maximal Cohen-Macaulay, as needed. We then apply Theorem \ref{acstgeneralization} with $M=\Omega_{R}$ and $N={\rm Der}(R)$ to conclude that $R$ is regular. \end{proof}

\begin{corollary}\label{best-reg2} Let $R$ be as in $($\ref{two-classes}$)$. Assume $R$ is Cohen-Macaulay of dimension $d\geq 2$ which is locally regular in codimension $n$, for an integer $n$ such that $1\leq n\leq d-1$. Suppose $\Ext^j_R({\rm Der}(R)^*, R)=0$ for all $j=1, \ldots, d$ and  $\Ext^i_R({\rm Der}(R), {\rm Der}(R))=0$ for all $i=n, \ldots, d-1$. In addition, assume  that either $\Omega_{R}$ is reflexive or $n\geq 2$. Then, $R$ is regular.
\end{corollary}
\begin{proof} Note ${\rm Der}(R)$ is locally free on $X^n(R)$. Moreover, being a dual over a normal domain, ${\rm Der}(R)$ satisfies $(S_2)$ and (equivalently) must be reflexive. By Theorem \ref{acstgeneralization} with $M={\rm Der}(R)$ and $N=R$, we conclude that ${\rm Der}(R)$ is free. Now, in the case of $\Omega_{R}$ being reflexive, $R$ must be regular. Finally, if $n\geq 2$ (i.e., $R$ is locally regular in codimension at least 2) then we have ${\rm codim}\,{\rm Sing}\,R\geq 3$, where ${\rm Sing}\,R$ stands for the Zariski-closed set formed by the ${\mathfrak p}\in {\rm Spec}\,R$ such that $R_{\mathfrak p}$ is not regular; now the regularity of $R$ follows by \cite[Corollary, p.\,318]{F}.
\end{proof}



    


The next regularity characterizations deal in particular with the transpose ${\rm Tr}\,\Omega_{R}$, which, for $R$ as in (\ref{two-classes}), can be realized as the cokernel of the map of free $R$-modules induced by the Jacobian matrix of some (any) minimal generating set of $I$. 

\begin{corollary}\label{reg-Tr}  Let $R$ be as in $($\ref{two-classes}$)$. Assume that $R$ is Cohen-Macaulay of dimension $d\geq1$. Suppose any one of the following sets of conditions:
\begin{itemize}
\item [(i)] $\Ext^j_R(\Omega_{R}, R)=0$ for all $j=1, \ldots, d$, and $\Ext^i_R(\Tr \Omega_{R}, {\rm Der}(R))=0$ for all $i=1, \ldots, d+1$;

\item [(ii)] $\Omega_{R}$ is maximal Cohen-Macaulay, and $\Ext^i_R(\Tr \Omega_{R}, {\rm Der}(R, \omega_R))=0$ for all $i=1, \ldots, d+1$.
\end{itemize}
Then, $R$ is regular.
\end{corollary}
\begin{proof} For $(i)$ (resp. $(ii)$), apply Proposition \ref{freenesstr} with $N=R$ (resp. Corollary \ref{trfreenesscriterion}) and take $M=\Omega_R$.
\end{proof}


The corollary below deals with isolated singularities. For completeness, one instance where this property takes place (with $R$ as in (\ref{two-classes})) is when $\Omega_{R}$ is maximal Cohen-Macaulay and $R$ has only finitely many indecomposable non-isomorphic maximal Cohen-Macaulay modules of multiplicity at most $\nu \cdot \varepsilon$, where $\nu$ and $\varepsilon$ denote respectively the embedding dimension and the multiplicity of the local ring $R$ (see \cite[Corollary 4]{H-L-0}).

\begin{corollary}\label{Omega-Der}  Let $R$ be as in $($\ref{two-classes}$)$. Assume that $R$ is a Cohen-Macaulay isolated singularity of dimension $d\geq 1$. Suppose any one of the following sets of conditions:
\begin{itemize}
\item [(i)] $\Ext^j_R(\Omega_{R}, R)=\Ext^j_R(\Omega_{R}, \Omega_{R}\otimes_R\omega_R)=0$ for all $j=1, \ldots, d$, and $$\Ext^{d+1}_R(\Tr \Omega_{R}, {\rm Der}(R))=0  \quad (\mbox{e.g., if} \, \, \,  {\rm pd}_R{\rm Der}(R)<\infty \, \, \, \mbox{or} \, \, \,  {\rm id}_R{\rm Der}(R)<\infty).$$


\item [(ii)] $\Omega_{R}$ is maximal Cohen-Macaulay, $\Ext^j_R(\Omega_{R}, \Omega_{R})=0$ for all $j=1, \ldots, d$, and $$\Ext^{d+1}_R(\Tr \Omega_{R}, {\rm Der}(R, \omega_R))=0 \quad (\mbox{e.g., if} \, \, \,  {\rm pd}_R{\rm Der}(R)<\infty \, \, \, \mbox{or} \, \, \,  {\rm id}_R{\rm Der}(R, \omega_R)<\infty).$$
\end{itemize}
Then, $R$ is regular.
\end{corollary}
\begin{proof} Because $R$ is an isolated singularity, $\Omega_{R}$ is a vector bundle. In the case (i) (resp. (ii)), use Corollary \ref{NFlocus} with $N=R$ (resp. Corollary \ref{mcmvectorbundle}) and take $M=\Omega_R$.
\end{proof}

 
From Corollary \ref{reg-Tr} and Corollary \ref{Omega-Der}, we record an immediate byproduct in the Gorenstein case.

\begin{corollary}\label{reg-Tr-2}  Let $R$ be as in $($\ref{two-classes}$)$. Assume that $R$ is Gorenstein of dimension $d\geq1$ and that $\Omega_{R}$ is maximal Cohen-Macaulay. Suppose any one of the following sets of conditions:
\begin{itemize}
\item [(i)] $\Ext^j_R(\Tr \Omega_{R}, {\rm Der}(R))=0$ for all $j=1, \ldots, d+1$;
\item [(ii)] $R$ is an isolated singularity, $\Ext^j_R(\Omega_{R}, \Omega_{R})=0$ for all $j=1, \ldots, d$, and $\Ext^{d+1}_R(\Tr \Omega_{R}, {\rm Der}(R))=0$.
\end{itemize}
Then, $R$ is regular.
\end{corollary}

Next, we consider some long-standing differential conjectures in characteristic zero. We let $R$ be as in (\ref{two-classes}). The first one is the so-called {\it strong Zariski-Lipman conjecture}. Recall that the classical Zariski-Lipman conjecture asserts that $R$ must be regular if ${\rm Der}(R)$ is free (see \cite{Lip}). As a homological version of it, the Herzog-Vasconcelos conjecture predicts that ${\rm Der}(R)$ is free if  ${\rm pd}_R{\rm Der}(R)<\infty$. Combining them yields the following harder problem (see \cite{Her} and \cite{VC} for further details and references on these conjectures).

\begin{conjecture}[Zariski-Lipman-Herzog-Vasconcelos]\label{SZL}
Let $R$ be as in $($\ref{two-classes}$)$. If  ${\rm pd}_R{\rm Der}(R)<\infty$, then $R$ is regular.
\end{conjecture}

According to \cite[Satz 9.1]{S-S}, if ${\rm pd}_R{\rm Der}(R)<\infty$ then $R$ is necessarily a normal domain, which in particular settles the case of curves. It is also worth mentioning that the (classical) Zariski-Lipman conjecture remains open essentially for surfaces, the critical unsettled case being when $R$ is a non-graded two-dimensional Gorenstein ring. As to the Herzog-Vasconcelos conjecture, the critical open situations are when $R$ is a hypersurface ring or ${\rm depth}\,R=3$. We refer to \cite[Section 5]{VC}.


In the result below, the dimension of the local ring $R$ is denoted $d$ and is assumed to be positive. Recall, for item $(iii)$, that a {\it Gorenstein module} over $R$ is a maximal Cohen–Macaulay $R$-module of finite injective dimension.  

\begin{corollary} The strong Zariski-Lipman conjecture $($\ref{SZL}$)$ is true in case $R$ is a Cohen–Macaulay isolated singularity and, in addition, any one of the following situations holds:
\begin{itemize}
\item [(i)] $\cidim_R\Omega_{R}<\infty$ $($or alternatively $\cidim_RN^{*}<\infty$$)$, $\Ext^j_R(\Omega_{R}, N)=0$ for all $j=1, \ldots, d$ and ${\rm Tor}^R_i(\Omega_{R}, N^{*})=0$ for all $i\gg 0$, for some maximal Cohen-Macaulay $R$-module $N$;

\item [(ii)] $($$R$ not necessarily an isolated singularity.$)$ $\Ext^j_R(\Omega_{R}, N)=0$ for all $j>0$ and $\widehat{{\rm Tor}}^R_i(\Omega_{R}, N^{*})=0$ for all $i\in {\mathbb Z}$, for some maximal Cohen-Macaulay $R$-module $N$;



\item [(iii)] $\Omega_{R}$ is maximal Cohen-Macaulay and $\Ext^j_R(\Omega_{R}, \Omega_{R})=0$ for all $j=1, \ldots, d$ $($e.g., if\, $\Omega_{R}$ is a Gorenstein module$)$.
\end{itemize}
\end{corollary}
\begin{proof} Since $R$ is Cohen-Macaulay and ${\rm pd}_R{\rm Der}(R)<\infty$, $R$ must be Gorenstein
(see \cite[Theorem 4.5 and the comment after the proof]{Her}). Now if $(i)$ or $(ii)$ holds,
Corollary \ref{depthformula2} ensures that the module $\Omega_R$ is maximal Cohen-Macaulay. Thus, as $R$ is Gorenstein, $\Omega_R$ is reflexive and its dual ${\rm Der}(R)$ is maximal Cohen-Macaulay as well. It follows that ${\rm Der}(R)$ is free, hence so is $\Omega_R$, as needed. Finally, for the situation $(iii)$, we apply Corollary \ref{reg-Tr-2}$(ii)$. It remains to show that $(iii)$ holds if the module $\Omega_R$ is Gorenstein. Because $\Omega_R$ is maximal Cohen-Macaulay and ${\rm id}_R\Omega_R<\infty$, we can use \cite[Exercise 3.1.24]{BH} to get $${\rm sup}\{k\geq 0 \, \mid \, \Ext^k_R(\Omega_{R}, \Omega_{R})\neq 0\}=d-{\rm depth}\,\Omega_R=0,$$ so that $\Ext^j_R(\Omega_{R}, \Omega_{R})=0$ for all $j>0$.
\end{proof}

Motivated by the strong Zariski-Lipman conjecture $($\ref{SZL}$)$ and some of the results above, we propose the following problem, which in the Gorenstein case matches the statement of $($\ref{SZL}$)$.

\begin{question}\rm Let $R$ be as in $($\ref{two-classes}$)$. Suppose $R$ is Cohen-Macaulay with canonical module $\omega_R$.  If ${\rm id}_R{\rm Der}(R)<\infty$ or ${\rm id}_R{\rm Der}(R, \omega_R)<\infty$, must $R$ be regular?
\end{question}

At this point we mention, for completeness, the following conjecture also involving derivation modules but leading to different ring-theoretic properties. The $\gdim$ case of the problem has been settled affirmatively in \cite[Proposition 4.14]{Asg} for (graded) Cohen-Macaulay rings of minimal multiplicity. 

\begin{conjecture} (\cite[Conjecture 3.12]{MN})\label{C} Let $R$ be as in $($\ref{two-classes}$)$. Assume that ${\gdim}_R{\rm Der}(R)<\infty$ $($resp. ${\cidim}_R{\rm Der}(R)<\infty$$)$. Then, $R$ is Gorenstein $($resp. a complete intersection ring$)$.
\end{conjecture}

Another differential conjecture we wish to consider is the long-standing {\it Berger's conjecture}, raised in \cite{Ber}. We also refer to the survey \cite{BeSur} and the references suggested there.

\begin{conjecture}[Berger]\label{Berger}
Let $R$ be as in $($\ref{two-classes}$)$. Assume that $R$ is reduced and one-dimensional. If $\Omega_{R}$ is torsionfree, then $R$ is regular.
\end{conjecture}

\begin{corollary} Berger's   conjecture $($\ref{Berger}$)$ is true in any one of the following situations:
\begin{itemize}
\item [(i)] $\Ext^1_R({\rm Tr}\,\Omega_{R}, {\rm Der}(R, \omega_{R}))=\Ext^2_R({\rm Tr}\,\Omega_{R}, {\rm Der}(R, \omega_{R}))=0$;

\item [(ii)] $\Ext^1_R(\Omega_{R}, \Omega_{R})=\Ext^2_R({\rm Tr}\,\Omega_{R}, {\rm Der}(R, \omega_{R}))=0$.
\end{itemize}
\end{corollary}
\begin{proof} If $(i)$ (resp. $(ii)$) holds, we apply Corollary \ref{HIWconjectureext12} (resp. Corollary \ref{mcmvectorbundle2}) with $M=\Omega_R$.
\end{proof}

Finally notice that, because $d=1$, the condition $\Ext^2_R({\rm Tr}\,\Omega_{R}, {\rm Der}(R, \omega_{R}))=0$ is satisfied whenever ${\rm id}_R{\rm Der}(R, \omega_R)<\infty$.


\subsection{Complete intersection local rings} In this subsection, we consider local rings $R$ of the form
\begin{equation}\label{quot-of-reg} R=S/I, \quad (S, \mathfrak{M}) \, \, \mbox{a regular local ring,} 
\end{equation} where $I\subset \mathfrak{M}$ is an ideal of $S$. We are interested in characterizing when $I$ can be generated by an $S$-sequence; recall the classical fact that (since ${\rm pd}_SI<\infty$) this property is equivalent to the freeness of the conormal module $I/I^2$ over $R$; see \cite{V1}. We write\, ${\rm N}_R=(I/I^2)^*$ for the normal module of $R$.

\begin{corollary}\label{best-ci} Let $R$ be as in $($\ref{quot-of-reg}$)$. Assume $R$ is Cohen-Macaulay of dimension $d\geq 2$ and locally a complete intersection in codimension $n$, for an integer $n$ such that $1\leq n\leq d-1$. Assume that $I/I^2$ is reflexive. Suppose any one of the following sets of conditions:  \begin{itemize}
    \item [(i)] $\Ext^j_R({\rm N}_R, R)=0$ for all $j=1, \ldots, d$, and $\Ext^i_R(I/I^2, I/I^2)=0$ for all $i=n, \ldots, d-1$;

    \item [(ii)] $\Ext^j_R(I/I^2, R)=0$ for all $j=1, \ldots, d-2$ $($if $d\geq 3$$)$, $\Ext^i_R({\rm N}_R, {\rm N}_R)=0$ for all $i=1, \ldots, d$, and $\Ext^k_R(I/I^2, {\rm End}_R({\rm N}_R))=0$ for all $k=n, \ldots, d-1$;
    
   \item [(iii)] $\Ext^j_R(I/I^2, R)=0$ for all $j=1, \ldots, d$, and  $\Ext^i_R({\rm N}_R, {\rm N}_R)=0$ for all $i=n, \ldots, d-1$.
   
   
    \end{itemize} Then, $I$ can be generated by an $S$-sequence.
\end{corollary}
\begin{proof} Since $R$ is locally a complete intersection in codimension $n$, the $R_{\mathfrak p}$-module $(I/I^2)_{\mathfrak p}$ is free for all  ${\mathfrak p} \in X^n(R)$. Moreover, as in particular $R$ is locally Gorenstein in codimension 1, the condition of $I/I^2$ being reflexive is equivalent to $I/I^2$ satisfying $(S_2)$. If we assume that $(i)$ holds, then by applying Theorem \ref{acstgeneralization} with $M=I/I^2$ and $N=R$ we obtain that $I/I^2$ is free, as needed.

Now, suppose $(ii)$ holds. We claim that ${\rm N}_R$ is maximal Cohen-Macaulay. If $d=2$ then this holds since $R$ is Cohen-Macaulay and ${\rm N}_R$ is a dual over $R$. Now assume $d\geq 3$. By Remark \ref{MCMdual}$(i)$, the condition $\Ext^j_R(I/I^2, R)=0$ for all $j=1, \ldots, d-2={\rm max}\{1, d-2\}$ implies that ${\rm N}_R$ is maximal Cohen-Macaulay, as needed. We then apply Theorem \ref{acstgeneralization} with $M=I/I^2$ and $N={\rm N}_R$ to conclude that $I/I^2$ is free.

Finally, let us assume $(iii)$. The $R$-module ${\rm N}_R$ is locally free on $X^n(R)$, and moreover, being a dual over a Cohen-Macaulay ring which is locally Gorenstein in codimension 1, it satisfies $(S_2)$ and is reflexive. By Theorem \ref{acstgeneralization} with $M={\rm N}_R$ and $N=R$, we get that ${\rm N}_R$ is free. Since $I/I^2$ is reflexive, $I$ can be generated by a regular sequence. \end{proof}


In the next corollary, $R$ is no longer assumed to be locally a complete intersection in certain codimension. Instead, we deal with the Auslander transpose and the canonical dual of $I/I^2$.

\begin{corollary}  Let $R$ be as in $($\ref{quot-of-reg}$)$. Assume $R$ is Cohen-Macaulay of dimension $d\geq1$. Suppose any one of the following sets of conditions:
\begin{itemize}
\item [(i)] $\Ext^j_R(I/I^2, R)=0$ for all $j=1, \ldots, d$, and $\Ext^i_R(\Tr I/I^2, {\rm N}_R)=0$ for all $i=1, \ldots, d+1$;

\item [(ii)] $I/I^2$ is maximal Cohen-Macaulay, and $\Ext^i_R(\Tr I/I^2, (I/I^2)^{\dagger})=0$ for all $i=1, \ldots, d+1$.
\end{itemize}
Then, $I$ can be generated by an $S$-sequence.
\end{corollary}
\begin{proof} For $(i)$ (resp. $(ii)$), apply Proposition \ref{freenesstr} with $N=R$ (resp. Corollary \ref{trfreenesscriterion}) and take $M=I/I^2$.
\end{proof}

\begin{corollary}\label{Conormal-N}  Let $R$ be as in $($\ref{quot-of-reg}$)$. Assume that $R$ is Cohen-Macaulay of dimension $d\geq 1$ and locally a complete intersection on its punctured spectrum. Suppose any one of the following sets of conditions:
\begin{itemize}
\item [(i)] $\Ext^j_R(I/I^2, R)=\Ext^j_R(I/I^2,  I/I^2\otimes_R\omega_R)=0$ for all $j=1, \ldots, d$, and $$\Ext^{d+1}_R(\Tr I/I^2, {\rm N}_R)=0  \quad (\mbox{e.g., if} \, \, \,  {\rm pd}_R{\rm N}_R<\infty \, \, \, \mbox{or} \, \, \,  {\rm id}_R{\rm N}_R<\infty).$$


\item [(ii)] $I/I^2$ is maximal Cohen-Macaulay, $\Ext^j_R(I/I^2, I/I^2)=0$ for all $j=1, \ldots, d$, and $$\Ext^{d+1}_R(\Tr I/I^2, (I/I^2)^{\dagger})=0 \quad (\mbox{e.g., if} \, \, \,  {\rm pd}_R{\rm N}_R<\infty \, \, \, \mbox{or} \, \, \,  {\rm id}_R(I/I^2)^{\dagger}<\infty).$$
\end{itemize}
Then, $I$ can be generated by an $S$-sequence.
\end{corollary}
\begin{proof} Since $R$ is locally a complete intersection on the punctured spectrum, $I/I^2$ is a vector bundle. In the case (i) (resp. (ii)), use Corollary \ref{NFlocus} with $N=R$ (resp. Corollary \ref{mcmvectorbundle}) and take $M=I/I^2$.
\end{proof}

It is worth mentioning that the property of $R$ (taken as in (\ref{quot-of-reg})) being locally a complete intersection on its punctured spectrum takes place if, e.g., $I/I^2$ is maximal Cohen-Macaulay and $R$ is Cohen–Macaulay having only finitely many indecomposable non-isomorphic maximal Cohen-Macaulay
modules of rank at most the minimal number of generators of $I$ (see \cite[Corollary 6]{H-L-0}).

From the  Gorenstein case of the two corollaries above, we obtain the following characterizations.

\begin{corollary}\label{ci-Tr-2}  Let $R$ be as in $($\ref{quot-of-reg}$)$. Assume that $R$ is Gorenstein of dimension $d\geq1$ and that $I/I^2$ is maximal Cohen-Macaulay. Suppose any one of the following sets of conditions:
\begin{itemize}
\item [(i)] $\Ext^j_R(\Tr I/I^2, {\rm N}_R)=0$ for all $j=1, \ldots, d+1$;
\item [(ii)] $R$ is locally a complete intersection on its punctured spectrum, $\Ext^j_R(I/I^2, I/I^2)=0$ for all $j=1, \ldots, d$, and $\Ext^{d+1}_R(\Tr I/I^2, {\rm N}_R)=0$.
\end{itemize}
Then, $I$ can be generated by an $S$-sequence.
\end{corollary}

For Gorenstein ideals of height 3, we have simpler statements.

\begin{corollary}\label{codim3}  Let $R$ be as in $($\ref{quot-of-reg}$)$. Assume that $R$ is Gorenstein of dimension $d\geq1$ and that ${\rm ht}\,I=3$ . Suppose any one of the following sets of conditions:
\begin{itemize}
\item [(i)] $R$ is locally a complete intersection in codimension $0$ $($e.g., if\, $R$ is reduced$)$, and $\Ext^j_R(I/I^2, {\rm N}_R)=0$ for all $j=1, \ldots, d+1$;
\item [(ii)] $R$ is locally a complete intersection on its punctured spectrum, $\Ext^j_R(I/I^2, I/I^2)=0$ for all $j=1, \ldots, d$, and $\Ext^{d+1}_R(I/I^2, {\rm N}_R)=0$.
\end{itemize}
Then, $I=(f, g, h)$ for some $S$-sequence $\{f, g, h\}\subset \mathfrak{M}$.
\end{corollary}
\begin{proof} Under such hypotheses, $I/I^2$ is maximal Cohen-Macaulay (see \cite{Her0}, \cite{Hu-U}). Moreover, the well-known Buchsbaum-Eisenbud structure theorem (see \cite{B-E}) guarantees that $I$ admits an $S$-free presentation of the form $\xymatrix@=1em{F\ar[r] & F\ar[r] & I\ar[r] & 0}$ defined by an alternating map $\varphi \in {\rm End}_S(F)$. Then, the conormal module $I/I^2=I\otimes_SR$ has an $R$-free presentation $$\xymatrix@=1em{F/IF\ar[r] & F/IF\ar[r] & I/I^2\ar[r] & 0}$$ defined by $\overline{\varphi}=\varphi \otimes {\rm Id}_R$. Set ${\rm Hom}_R(\overline{\varphi}, R)=\overline{\varphi}^*$. Since $\varphi$ is alternating, ${\rm im}\, \overline{\varphi}^*= {\rm im}\, \overline{\varphi}$ and therefore we can write ${\rm Tr}\,I/I^2={\rm coker}\, \overline{\varphi}^*= {\rm coker}\, \overline{\varphi}=I/I^2$. Now the result follows by Corollary \ref{ci-Tr-2}. \end{proof}

We close the subsection by raising the following problem.

\begin{question}\rm Let $R$ be as in $($\ref{quot-of-reg}$)$. Suppose that $R$ is Cohen-Macaulay with a canonical module and that $R$ contains a field of characteristic zero.  If ${\rm id}_RI/I^2<\infty$, ${\rm id}_R{\rm N}_R<\infty$, or ${\rm id}_R(I/I^2)^{\dagger}<\infty$ (resp. ${\cidim}_RI/I^2<\infty$, ${\cidim}_R{\rm N}_R<\infty$, or ${\cidim}_R(I/I^2)^{\dagger}<\infty$), can $I$ be generated by an $S$-sequence?
\end{question}

\subsection{Gorenstein local rings} In this last subsection we deal with the Gorenstein property. First, by taking $M=\omega_R$ in Corollary \ref{trfreenesscriterion} we immediately derive the following characterization.



\begin{corollary}\label{Gor-charact}
Let $R$ be a Cohen-Macaulay local ring of dimension $d$ with canonical module $\omega_R$. Then, $R$ is Gorenstein if $$\Ext^j_R(\Tr \omega_R, R)=0 \quad \mbox{for \,all} \quad j=1, \ldots, d+1.$$ 
\end{corollary}

Below we detect a situation where the vanishing of a single Ext module (to wit, $\Ext^{d+1}_R(\Tr \omega_R, R)$) suffices to ensure Gorensteiness. The additional condition is that $R$ be locally Gorenstein on its punctured spectrum; this occurs, for instance, if $R$ has only finitely many indecomposable non-isomorphic maximal Cohen-Macaulay modules of multiplicity at most $\tau \cdot \varepsilon$, where $\tau$ and $\varepsilon$ denote respectively the Cohen-Macaulay type and the multiplicity of the local ring $R$ (see \cite[Corollary 5]{H-L-0}).

\begin{corollary}\label{Gor-charact-loc}
Let $R$ be a Cohen-Macaulay local ring of dimension $d\geq 1$ with canonical module $\omega_R$. Suppose $R$ is locally Gorenstein on its punctured spectrum. Then, $R$ is Gorenstein if $$\Ext^{d+1}_R(\Tr \omega_R, R)=0 \quad (\mbox{e.g., if} \, \, \,  {\rm pd}_R\omega_R^*<\infty).$$ 
\end{corollary}
\begin{proof} As $R$ is locally Gorenstein on its punctured spectrum, the (maximal Cohen-Macaulay) module $\omega_R$ is a vector bundle. Moreover, since $\Ext^{j}_R(\omega_R, \omega_R)=0$ for all $j>0$ and $\omega_R^{\dagger}\cong R$, the result follows by Corollary \ref{mcmvectorbundle}. \end{proof}


In the Artinian case, Corollary \ref{Gor-charact} ensures that $R$ is Gorenstein if and only if $\Ext^1_R(\Tr \omega_R, R)=0$. Now if $R$ is  one-dimensional and reduced, then using Corollary \ref{Gor-charact-loc} we get that $R$ is Gorenstein if and only if $\Ext^2_R(\Tr \omega_R, R)=0$ if and only if $\omega_R^*$ is free. In the next result, we treat -- while revealing some extra information -- the case $d=2$.


     


\begin{corollary}\label{Gor-charact-2}
Let $R$ be a two-dimensional Cohen-Macaulay local ring with canonical module $\omega_R$ which is locally Gorenstein on its punctured spectrum $($e.g., if $R$ is normal$)$. The following assertions are equivalent:
\begin{itemize}
    \item [(i)] $R$ is Gorenstein;
    \item [(ii)] $\Ext^3_R(\Tr \omega_R, R)=0$;
     \item [(iii)] ${\rm pd}_R\omega_R^*<\infty$;
     
     
      \item [(iv)] ${\cidim}_R\omega_R^*<\infty$ and ${\rm pd}_R(\omega_R^*)^{\dagger}<\infty$.
\end{itemize}
\end{corollary}
\begin{proof} The equivalences $(i)\Leftrightarrow(ii)\Leftrightarrow(iii)$ follow by Corollary \ref{Gor-charact-loc}. Note  $(i)\Rightarrow(iv)$ is trivial. Now suppose $(iv)$ holds. Being a dual over a Cohen-Macaulay local ring with $d=2$, the module $\omega_R^*$ is maximal Cohen-Macaulay.
By (\ref{CI-dim-form}) we get  ${\cidim}_R\omega_R^*=0$. On the other hand,
Lemma \ref{tor0} gives $${\rm Tor}_j^R(\omega_R^*, (\omega_R^*)^{\dagger})=0 \quad \mbox{for \,all} \quad j>0.$$ Now Proposition \ref{cidimtor} yields that $\omega_R^*$ is free, which as we already know means that $R$ is Gorenstein.
\end{proof}



In higher dimension, we have the corollary below. It is worth recalling that, if $R$ is a Cohen-Macaulay local ring with canonical module $\omega_R$ and dimension $d$, such that $R$ is locally Gorenstein in codimension 0, then $R$ is Gorenstein if and only if  $\Ext^j_R(\omega_R, R)=0$ for all $j=1, \ldots, d$ (see \cite[Corollary 2.2]{HH}).

\begin{corollary}\label{Gor-charact-d}
Let $R$ be a Cohen-Macaulay local ring with canonical module $\omega_R$ which is locally Gorenstein in codimension $1$ $($e.g., if $R$ is normal$)$ and dimension $d\geq 3$. The following assertions are equivalent:
\begin{itemize}
    \item [(i)] $R$ is Gorenstein;
    \item [(ii)] ${\rm pd}_R\omega_R^*<\infty$ and $\Ext^j_R(\omega_R, R)=0$ for all $j=1, \ldots, d-2$;
     
      
      \item [(iii)] ${\cidim}_R\omega_R^*<\infty$, $\Ext^j_R(\omega_R, R)=0$ for all $j=1, \ldots, d-2$, and ${\rm pd}_R(\omega_R^*)^{\dagger}<\infty$;
      
      \item [(iv)] $\Ext^j_R(\omega_R^*, R)=0$ for all $j=1, \ldots, d$.
\end{itemize}
\end{corollary}
\begin{proof} It is clear that $(i)$ implies $(ii)$, $(iii)$, and $(iv)$. Assume $(ii)$ holds. By Remark \ref{MCMdual}$(i)$, the vanishing of ${\rm Ext}_R^j(\omega_R, R)$ whenever $j = 1, \ldots, d - 2$ forces $\omega_R^*$ to be maximal Cohen-Macaulay, and hence it must be free. Because $R$ is locally Gorenstein in codimension $1$, it is easy to see that $\omega_R$ is reflexive (use \cite[Proposition 1.4.1(b)]{BH}). It
follows that $R$ is Gorenstein. Now suppose $(iii)$ holds. In particular, $\omega_R^*$ is maximal Cohen-Macaulay
and therefore, by using (\ref{CI-dim-form}), Lemma \ref{tor0} and Proposition \ref{cidimtor} as in the proof of Corollary \ref{Gor-charact-2}, we obtain that $\omega_R^*$ is free, hence $R$ is Gorenstein by the reflexivity of $\omega_R$. Finally, assume $(iv)$. Note $\omega_R$ is locally free on $\X^1(R)$. Furthermore, we have $${\rm Ext}_R^i(\omega_R, \omega_R^{**})\cong {\rm Ext}_R^i(\omega_R, \omega_R)=0 \quad \mbox{for \,all} \quad i>0.$$ Applying Theorem \ref{acstgeneralization} with $n=1$, $M=\omega_R$ and $N=R$, we conclude that $\omega_R$ is free.
\end{proof}



\begin{remark}\rm $(i)$ The equivalence between items $(i)$ and $(iv)$ also holds in case $d=2$.

\medskip

$(ii)$ Our motivation for considering complete intersection dimension in some of the items came up from the fact that $R$ is Gorenstein if and only if $\cidim_R\omega_R<\infty$ (see \cite[Proposition 1.2]{JL}). Notice that we can recover this result by applying our Proposition \ref{cidimtor} with $M=\omega_R$.

\end{remark}

In virtue of the results above, it seems plausible to raise the following concluding question, part of which can be regarded as an analogue of conjecture (\ref{SZL}) for the Gorenstein property, i.e., with $\omega_R$ in place of $\Omega_{R}$.

\begin{question}\rm Let $R$ be a Cohen-Macaulay local ring with canonical module $\omega_R$.  If ${\rm pd}_R\omega_R^*<\infty$ or ${\rm pd}_R(\omega_R^*)^{\dagger}<\infty$ (resp. ${\cidim}_R\omega_R^*<\infty$ or ${\cidim}_R(\omega_R^*)^{\dagger}<\infty$), must $R$ be Gorenstein?
\end{question}



\bigskip

\noindent{\bf Acknowledgements.} 
The first-named author was supported by a CAPES Doctoral Scholarship. The second-named author was partially supported by the CNPq-Brazil grants 301029/2019-9 and 406377/2021-9.

\bigskip

\noindent{\bf Data availability statement.} Data sharing not applicable to this article as no datasets were generated or analysed during the current study.

\bigskip

\noindent{\bf Conflict of interest.} On behalf of all authors, the corresponding author states that there is no conflict of interest.

\end{document}